\documentclass[reqno]{amsart}
\usepackage[T1]{fontenc}
\usepackage{yhmath,mathrsfs,amsthm,amsmath,amssymb,amsfonts,enumerate,lipsum,appendix,mathtools}
\usepackage{bbold}
\usepackage[mathscr]{euscript}
\usepackage[normalem]{ulem} 
\usepackage{graphicx}
\usepackage{pgf,tikz}
\usepackage{tkz-euclide}
\usepackage{graphicx}
\usepackage{subcaption}
\usetikzlibrary{shapes.geometric}
\usetikzlibrary{arrows,hobby}
\usepackage{enumitem}
\usepackage[english]{babel}

\allowdisplaybreaks

\usepackage[sort,numbers]{natbib}
\usepackage{hyperref}
\hypersetup{
	colorlinks=true,
	linkcolor=blue,
	filecolor=magenta,      
	urlcolor=cyan,
	citecolor=cyan,
}
\usepackage[margin=1in]{geometry}
\allowdisplaybreaks
\newtheorem{theorem}{Theorem}[section]
\newtheorem{lemma}[theorem]{Lemma}
\newtheorem{proposition}[theorem]{Proposition}
\newtheorem{definition}[theorem]{Definition}
\theoremstyle{definition}
\newtheorem{example}{Example}

\newtheorem{corollary}[theorem]{Corollary}
\newtheorem{remark}[theorem]{Remark}
\numberwithin{equation}{section}
\usepackage[hyperpageref]{backref}
\begin{document}
	\title[Domain variations via strict Faber-Krahn type inequality]{Domain variations of  the first eigenvalue via a strict Faber-Krahn type inequality}
	\author[Anoop]{T. V. Anoop}
	\address{T. V. Anoop, Department of Mathematics, Indian Institute of Technology Madras, Chennai 600 036, India}
	\email{anoop@iitm.ac.in}
	\author[Ashok]{K. Ashok Kumar}
	\address{K. Ashok Kumar, Department of Mathematics, Indian Institute of Technology Madras, Chennai 600 036, India}
	\email{s.r.asoku@gmail.com}
	\thanks{Data sharing is not applicable to this article as no data sets were generated or analysed during the current study.}
	\begin{abstract}
		For $d\geq 2$ and $\frac{2d+2}{d+2} < p < \infty $, we prove a strict Faber-Krahn type inequality for the first eigenvalue $\lambda _1(\Omega )$ of the $p$-Laplace operator on a bounded Lipschitz domain $\Omega \subset \mathbb{R}^d$ (with mixed boundary conditions) under the polarizations. We apply this inequality to the obstacle problems on the domains of the form $\Omega \setminus \mathscr{O}$, where $\mathscr{O}\subset \subset \Omega $ is an obstacle. Under some geometric assumptions on $\Omega $ and $\mathscr{O}$, we prove the strict monotonicity of $\lambda _1 (\Omega \setminus \mathscr{O})$ with respect to certain translations and rotations of $\mathscr{O}$ in $\Omega $.
	\end{abstract}
	\maketitle
	\section{Introduction}\label{intro}
	In 1877, Lord Rayleigh~\cite{Rayleigh} conjectured that `\emph{the disk is the only planar domain that minimizes the first Dirichlet eigenvalue of the Laplace operator among all planar domains of fixed area.}' Nearly after 45 years, this conjecture was proved by Faber~\cite{Faber} and Krahn~\cite{Krahn} for the planar domains (in 1923), and it is extended for higher dimensional domains by Krahn~\cite{Krahn19261} (in 1925).  This result is known as the Faber-Krahn inequality which is also available for the first Dirichlet eigenvalue of the $p$-Laplace operator $\Delta _p$, defined by $\Delta _p u={\rm div}(|\nabla u|^{p-2}\nabla u)$ with $p\in (1,\infty )$, see for example \cite[page 191]{Polya-Szego} and \cite[II.4]{Kawohl85}. For a domain $\Omega \subset \mathbb{R}^{d}$, the Faber-Krahn inequality states that
	\begin{align}\label{faber-krahn}
		\lambda _1 (\Omega ^*) \leq \lambda _1 (\Omega ),
	\end{align}  
	where $\lambda _1(D)$ denotes the first Dirichlet eigenvalue of the $p$-Laplace operator  on a domain $D$ and $\Omega ^*$ is the open ball centred at the origin in $\mathbb{R}^{d}$ with the same Lebesgue measure as that of $\Omega $. If $\Omega $ is a ball, then the equality holds in~\eqref{faber-krahn}. The question \emph{`for which domains the strict inequality holds in~\eqref{faber-krahn}?'} is addressed in~\cite{DanersKennedy, Alvino-Ferone-Trombetti, Bhatta, Kesavan88,Anisa1}.
	
	\par  Noting that $\Omega ^*$ is the Schwarz symmetrization of $\Omega $, the inequality \eqref{faber-krahn} asserts that the first Dirichlet eigenvalue decreases under the Schwarz symmetrization. Next, we see that a similar result easily holds under the polarization as well. The polarization is one of the simplest rearrangements on $\mathbb{R}^d$ that was first introduced for sets by Wolontis~\cite{Wolontis}, and for functions by Ahlfors~\cite{Ahlfors73} (for $d=2$) and Baernstein and Taylor~\cite{Baernstein76} (for $d\geq 2$). We refer to \cite{Solynin12,Anoop-Ash-Kesh,Weth2010,Brock04,NirjanUjjalGhosh21} for further reading on polarizations and their applications. Now, we define the polarization of measurable sets and functions with respect to an open affine-halfspace in $\mathbb{R}^d$. Let $H$ be an open affine-halfspace in $\mathbb{R}^d$ (called a \emph{polarizer}), and let $\sigma _H$ be the reflection with respect to the boundary $\partial H$ in $\mathbb{R}^d$. We denote the set of all polarizers in $\mathbb{R}^d$ by $\mathscr{H}$. 
	\begin{definition}[Polarization]\label{defn-Polarization}
		Let $H\in \mathscr{H}$ and $\Omega \subseteq \mathbb{R}^d$. The polarization $P_H(\Omega )$ and the dual-polarization $P^H(\Omega )$ of $\Omega $ with respect to $H$ are defined as:
	\begin{align*}
		P_H(\Omega )&=\left[\left(\Omega \cup \sigma _H(\Omega )\right)\cap H \right]\cup \left[\Omega  \cap \sigma _H(\Omega )\right], \\
		P^H(\Omega )&=\left[\left(\Omega \cup \sigma _H(\Omega )\right)\cap H^\mathsf{c} \right] \cup \left[\Omega  \cap \sigma _H(\Omega )\right].
	\end{align*}
	For a measurable function $u:\mathbb{R}^d\longrightarrow \mathbb{R}$, the polarization $P_H(u)$ with respect to $H$ is defined as
	\begin{align*}
		P_H(u)(x) =
		\left\{
		\begin{aligned}
			&\max \left\{u(x),u(\sigma _H(x))\right\},&\mbox{ for } & x\in H,\\
			&\min \left\{u(x),u(\sigma _H(x))\right\},&\mbox{ for } & x\in \mathbb{R}^d\setminus H.
		\end{aligned}
		\right.
	\end{align*}
	Now, for $u:\Omega \longrightarrow \mathbb{R}$ let $\widetilde{u}$ be the zero extension of $u$ to $\mathbb{R}^d$. The polarization $P_H(u)$ is defined as the restriction of $P_H(\widetilde{u})$ to $P_H(\Omega )$.
	\end{definition}
	\begin{remark}
	The polarization of the sets and the functions satisfy the following relation:
	\begin{align*}
	    P_H(\mathbb{1}_{\Omega }) = \mathbb{1}_{P_H(\Omega )}, \mbox{ for any } \Omega \subseteq \mathbb{R}^d,
	\end{align*}
	where $\mathbb{1}_\Omega $ denotes the characteristic function of $\Omega $.
	\end{remark}
	In Figure~\ref{Figure-pol}, the dark shaded regions on the right side represent the polarization $P_H (\Omega )$ of $\Omega $ with respect to $H$. 
	\begin{figure}[hbt!]
		\centering
		\begin{tikzpicture}[scale=0.2]
			\begin{scope}[xshift=-36cm]
				\fill [gray!15] (-10, 0) rectangle (10,9);
				\draw [rotate around={-60:(1,0)}, fill=gray!75] (0, 0) ellipse (8cm and 4cm);
				\draw (-10,0) -- (10,0);
				\draw (-10,5) node[anchor=north west] {$\scriptstyle H$};
				\draw (1.75,-3) node[anchor=north west] {$\scriptstyle \Omega $};
			\end{scope}
			\fill [gray!15] (-10, 0) rectangle (10,9);
			\draw [rotate around = {60:(1,0)},fill=gray!75] (0, 0) ellipse (8cm and 4cm);
			\draw [rotate around = {-60:(1,0)}, fill=gray!75] (0, 0) ellipse (8cm and 4cm);
			\draw (-10,0) -- (10,0);
			\draw (-10,4) node[anchor=north west] {$\scriptstyle H$};
			\draw (-5.3,-4.5) node[anchor=north west] {$\scriptstyle \sigma_H(\Omega )$};
			\draw (2.5,-3.5) node[anchor=north west] {$\scriptstyle \Omega $};
			\begin{scope}
				\clip [rotate around = {60:(1,0)}] (0, 0) ellipse (8cm and 4cm);
				\clip [rotate around = {-60:(1,0)}] (0, 0) ellipse (8cm and 4cm);
				\fill [color=gray!130] (-8,8) rectangle (8,-8);
			\end{scope}
			\begin{scope}
				\clip (10,0) rectangle (-10,10);
				\clip [rotate around = {60:(1,0)}] (0,0) ellipse (8cm and 4cm);
				\fill [color=gray!130] (-8,8) rectangle (8,-8);
			\end{scope}
			\begin{scope}
				\clip (10,0) rectangle (-10,10);
				\clip [rotate around = {-60:(1,0)}] (0,0) ellipse (8cm and 4cm);
				\fill [color=gray!130] (-8,8) rectangle (8,-8);
			\end{scope}
			\draw (-2,3) node[anchor=north west] {$\scriptstyle P_H(\Omega )$};
			\begin{scope}[xshift=-36cm,yshift=-22cm]
				\fill [gray!15,rotate=22.5] (-10, 0) rectangle (10,8.5);
				\draw [fill=gray!75] (-6,-6) rectangle (6,6);
				\draw [rotate=22.5] (-10,0) -- (10,0);
				\draw (-11,2.5) node[anchor=north west] {$\scriptstyle H$};
				\draw (-8.5,-4) node[anchor =north west] {$\scriptstyle \Omega $};
			\end{scope}
			\begin{scope}[yshift=-22cm]
				\fill [gray!15,rotate=22.5] (-10, 0) rectangle (10,8.5);
				\draw [fill=gray!75] (-6,-6) rectangle (6,6);
				\draw [rotate=22.5] (-10,0) -- (10,0);
				\draw [fill=gray!75,rotate=45] (-6,-6) rectangle (6,6);
				\draw (-11,2.5) node[anchor=north west] {$\scriptstyle H$};
				\draw (-8.5,-4) node[anchor =north west] {$\scriptstyle \Omega $};
				\draw (-2,-8) node[anchor=north west] {$\scriptstyle \sigma _H(\Omega )$};
				\begin{scope}
					\clip (-6,-6) rectangle (6,6);
					\clip [rotate=22.5] (-10,0) rectangle (10,10);
					\fill [color=gray!130] (-10,-10) rectangle (10,10);
				\end{scope}
				\begin{scope}
					\clip [rotate=45](-6,-6) rectangle (6,6);
					\clip [rotate=22.5](-10,0) rectangle (10,10);
					\fill [color=gray!130] (-10,-10) rectangle (10,10);
				\end{scope}
				\begin{scope}
					\clip [rotate=45](-6,-6) rectangle (6,6);
					\clip (-6,-6) rectangle (6,6);
					\fill [color=gray!130] (-8,-8) rectangle (8,8);
				\end{scope}
				\draw (0,0) node[anchor=south] {$\scriptstyle P_H(\Omega )$};
			\end{scope}
		\end{tikzpicture}
		\caption{Polarization of an ellipse and a square.}
		\label{Figure-pol}
	\end{figure}
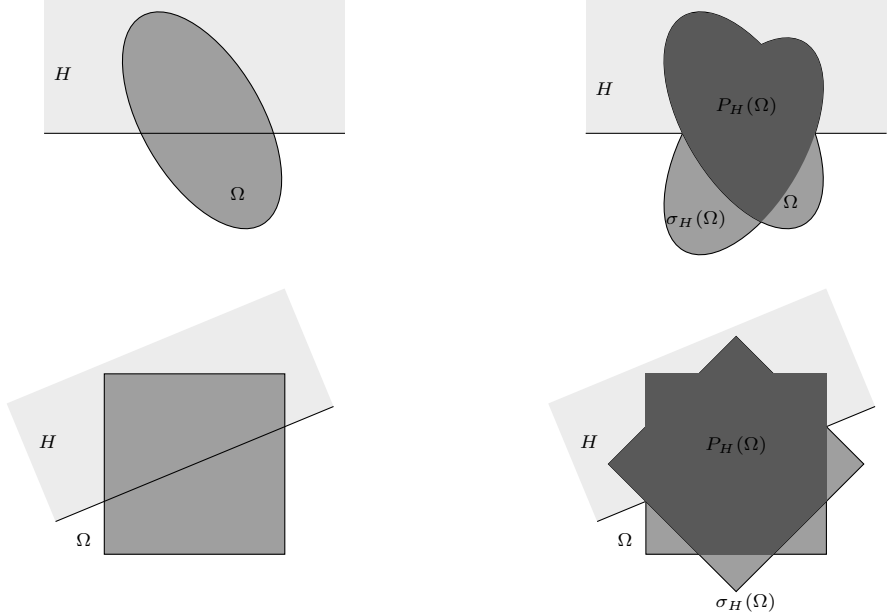 
	
	\noindent For $H\in \mathscr{H}$, the polarization $P_H$ is a rearrangement (preserves the inclusion order and the measure) on $\mathbb{R}^d$, see~\cite[Section~3]{BrockSolynin2000}. Further, $P_H$ takes an open set to an open set and a closed set to a closed set in $\mathbb{R}^d$. Throughout this article, we consider $p\in (1,\infty )$, unless otherwise specified. For a non-negative function $u\in W^{1,p}_0(\Omega )$ the polarization $P_H(u)\in W^{1,p}_0(P_H(\Omega ))$ and the norms are preserved, see~\cite[Corollary 5.1]{BrockSolynin2000}:
	\begin{align*}
		\left\| u \right\|_{\scriptscriptstyle p,\Omega } = \left\|P_H(u)\right\|_{\scriptscriptstyle p, P_H\Omega } \mbox{ and } \left\|\nabla u \right\|_{\scriptscriptstyle p, \Omega } = \left\|\nabla P_H(u) \right\|_{\scriptscriptstyle p, P_H\Omega }.
	\end{align*}
	Therefore, we have the equality in the P\'olya-Sz\"ego type inequality for the polarizations on $\mathbb{R}^d$. As an immediate consequence, the variational characterization of $\lambda _1(\Omega )$ yields the following Faber-Krahn type inequality: 
	\begin{align}\label{faber-krahn_pol}
		\lambda _{1}(P_H(\Omega ))\leq \lambda _{1}(\Omega ). 
	\end{align}
	Clearly, if $P_H(\Omega )=\Omega $ or $P_H(\Omega )=\sigma _H (\Omega )$ then the equality holds in~\eqref{faber-krahn_pol}. In this article, we identify the domains for which the strict inequality holds in~\eqref{faber-krahn_pol} for certain values of $p$. More precisely, we show that, if $p>\frac{2d+2}{d+2}$ and the equality holds in~\eqref{faber-krahn_pol} then $P_H(\Omega )=\Omega $ or $P_H(\Omega ) = \sigma _H (\Omega )$. We prove this result for the first eigenvalue of the $p$-Laplace operator with mixed boundary conditions on the multiply connected domains of the following form: 
	\vspace{2mm}
	\begin{enumerate}[label={$\bf (A_{\arabic*})$}]
		\setcounter{enumi}{-1}
		\item \label{hypothesis} $\Omega_{\rm out} \setminus \overline{\Omega _{\rm in}}\subset \mathbb{R}^d$ is a bounded Lipschitz domain with $\Omega_{\rm in} \subset \subset \Omega_{\rm out},$ and $\Omega _{\rm in} = \bigcup \limits_{j=1}^m \Omega _j,$ where $\Omega _j$ is simply connected and $\overline{\Omega _i} \cap \overline{\Omega _j}=\emptyset $ for $i, j\in \{1,2,\ldots , m\}$ with $i\neq j$. 
	\end{enumerate}
	\vspace{1mm}
	For $\Omega_{\rm out} \setminus \overline{\Omega _{\rm in}}$ as in~\ref{hypothesis}, we consider the following family of admissible polarizers
	\begin{align*}
	    \mathscr{H}_{\rm ad} :=\Big\{H\in \mathscr{H} : \sigma _H(\Omega _{\rm in}) \subset \subset \Omega _{\rm out}\Big\}.
	\end{align*}
	Since $\Omega _{\rm in} \subset \subset \Omega _{\rm out}$, the set $\mathscr{H}_{\rm ad}$ is always non-empty, and for $H\in \mathscr{H}_{\rm ad}$ we have (see Proposition~\ref{propo:pol-compliment1}):
	\begin{align*}
		P_H(\Omega_{\rm out} \setminus \overline{\Omega _{\rm in}}) = P_H(\Omega _{\rm out})\setminus P^H(\overline{\Omega _{\rm in}}) \mbox{ and } \partial P_H(\Omega_{\rm out} \setminus \overline{\Omega _{\rm in}}) = \partial P_H(\Omega _{\rm out}) \sqcup \partial P^H(\Omega _{\rm in}).
	\end{align*} 
	For $H\in \mathscr{H}_{\rm ad}$, we consider the following eigenvalue problems for the $p$-Laplace operator on both $\Omega $ and $P_H(\Omega )$ with mixed boundary conditions: 
	\vspace{0.5cm}
	
	\noindent
	\textbf{Neumann condition on $\partial \Omega _{\rm in}$:} 
	
	\noindent
	\begin{minipage}[t]{0.49\textwidth}
		\begin{align}\label{eigen1}
			\left.
			\begin{aligned}
				-\Delta _p u &=\nu |u|^{p-2}u \mbox{ in } \Omega_{\rm out} \setminus \overline{\Omega _{\rm in}},\\
				u &=0 \mbox{ on } \partial \Omega _{\rm out},\\
				\frac{\partial u}{\partial n} &=0 \mbox{ on } \partial \Omega _{\rm in};
			\end{aligned}
			\right\}
		\end{align}
	\end{minipage}
	\begin{minipage}[t]{0.49\textwidth}
		\begin{align}\label{eigen1-1}
			\left.
			\begin{aligned}
				-\Delta _p v &=\nu |v|^{p-2}v \mbox{ in } P_H(\Omega_{\rm out} \setminus \overline{\Omega _{\rm in}}),\\
				v & =0 \mbox{ on } \partial P_H(\Omega _{\rm out}),\\
				\frac{\partial v}{\partial n} &=0 \mbox{ on } \partial P^H(\Omega _{\rm in});
			\end{aligned}
			\right\}
		\end{align}
	\end{minipage}
	\vspace{0.5cm}
	
	\noindent
	\textbf{Neumann condition on $\partial \Omega _{\rm out}$:}
	
	\noindent
	\begin{minipage}[t]{0.49\textwidth}
		\begin{align}\label{eigen2}
			\left.
			\begin{aligned}
				-\Delta _p u &=\tau |u|^{p-2}u \mbox{ in } \Omega_{\rm out} \setminus \overline{\Omega _{\rm in}},\\
				u &=0 \mbox{ on } \partial \Omega _{\rm in},\\
				\frac{\partial u}{\partial n} &=0 \mbox{ on } \partial \Omega _{\rm out};
			\end{aligned}
			\right\}
		\end{align}
	\end{minipage}
	\begin{minipage}[t]{0.49\textwidth}
		\begin{align}
			\left.
			\begin{aligned}\label{eigen2-1}
				-\Delta _p v &=\tau |v|^{p-2}v \mbox{ in } P_H(\Omega_{\rm out} \setminus \overline{\Omega _{\rm in}}),\\
				v &=0 \mbox{ on } \partial P^H(\Omega _{\rm in}),\\
				\frac{\partial v}{\partial n} &=0 \mbox{ on } \partial P_H(\Omega _{\rm out}),
			\end{aligned}
			\right\}
		\end{align}
	\end{minipage} 
	
	\noindent where $\nu , \tau \in \mathbb{R}.$
	\vspace{2mm}
	
	\par The above eigenvalue problems can be collectively expressed as the following problem: 
	\begin{align}\label{eqn_intro_eigenZaremba1}\tag{$\mathcal{E}$}
		\left.
		\begin{aligned}
			-\Delta _p u &=\gamma |u|^{p-2}u \mbox{ in } \Omega ,\\
			u &=0 \mbox{ on } \Gamma _{\scriptscriptstyle D},\\
			\frac{\partial u}{\partial n} &=0 \mbox{ on } \Gamma _{\scriptscriptstyle N},
		\end{aligned}
		\right\}
	\end{align} 
	where $\Omega \subset \mathbb{R}^d$ is a bounded Lipschitz domain with $\partial \Omega = \Gamma _{\scriptscriptstyle N}\sqcup \Gamma _{\scriptscriptstyle D}$,  and $\gamma \in \mathbb{R}$. Let
	\begin{align*}
	    \mathcal{C}^{0,1}_{\scriptscriptstyle \Gamma _{D}}(\Omega ) \mathrel{\mathop:}= \Big\{v \mbox{ is a Lipschitz continuous function on }\Omega \mbox{ with } {\rm supp} \left(v\right)\cap \Gamma _{\scriptscriptstyle D}=\emptyset \Big\},
	\end{align*}
	and we consider the following Sobolev space:
	\begin{align*}
		W^{1,p}_{\scriptscriptstyle \Gamma _{\scriptscriptstyle D}}(\Omega )= \mbox{ the closure of } \mathcal{C}^{0,1}_{\scriptscriptstyle \Gamma _{\scriptscriptstyle D}}(\Omega ) \mbox{ in } W^{1,p}(\Omega ).
	\end{align*} 
	If $\Gamma _{\scriptscriptstyle N}=\emptyset $ (equivalently $\Gamma _{\scriptscriptstyle D}=\partial \Omega $) then $W^{1,p}_{\scriptscriptstyle \Gamma _{\scriptscriptstyle D}}(\Omega ) = W^{1,p}_0(\Omega )$. A real number $\gamma $ is said to be an eigenvalue of~\eqref{eqn_intro_eigenZaremba1} if there exists a non-zero function $u\in W^{1,p}_{\scriptscriptstyle \Gamma _{\scriptscriptstyle D}}(\Omega )$ such that
	\begin{align*}
		\int_{\Omega } |\nabla u|^{p-2}\nabla u \cdot \nabla v {\,\rm d}x -\gamma \int_{\Omega } |u|^{p-2}uv {\,\rm d}x =0 \mbox{ for every } v\in W^{1,p}_{\scriptscriptstyle \Gamma _{\scriptscriptstyle D}} (\Omega ),
	\end{align*}
	and the function $u$ is called as an eigenfunction corresponding to the eigenvalue $\gamma $. The standard variational arguments establish the existence of an infinite subset of eigenvalues tending to infinity (see~\cite[Proposition~A.1]{AAK}). The first eigenvalue $\gamma _{1} (\Omega )$ of~\eqref{eqn_intro_eigenZaremba1} is simple (the dimension of the eigenspace is one) and the corresponding eignfunctions have constant sign in $\Omega $ (see~\cite[Proposition~A.2]{AAK}). Moreover, the first eigenvalue $\gamma _{1} (\Omega )$ of~\eqref{eqn_intro_eigenZaremba1} has the following variational characterization:
	\begin{align*}
		\gamma _{1}(\Omega )=\inf \left\{\int _{\Omega }|\nabla u|^p {\,\rm d}x : u\in W^{1,p}_{\Gamma _{\scriptscriptstyle D}}(\Omega ) \mbox{ with } \int _{\Omega } |u|^p {\,\rm d}x =1 \right\}.
	\end{align*}
	Now, we state a Faber-Krahn type inequality for the first eigenvalues of the eigenvalue problems~\eqref{eigen1}~and~\eqref{eigen1-1}, and similarly for the first eigenvalues of the eigenvalue problems~\eqref{eigen2}~and~\eqref{eigen2-1}.  
	\begin{theorem}\label{thm~1.2}
		Let $p\in (1,\infty )$, $\Omega_{\rm out} \setminus \overline{\Omega _{\rm in}} \subset \mathbb{R}^d$ be a domain as given in~\ref{hypothesis}, and $H\in \mathscr{H}_{\rm ad}$.
		\begin{enumerate}[label={\rm (\roman*)}]
			\item If $\sigma_{\scriptscriptstyle H}(\Omega _{\rm in})=\Omega _{\rm in}$, then
			\begin{align}\label{f-k_pol1}
				\nu _{1}(P_H(\Omega_{\rm out} \setminus \overline{\Omega _{\rm in}})) \leq \nu _{1}(\Omega_{\rm out} \setminus \overline{\Omega _{\rm in}}).
			\end{align}
			\item If $\Omega_{\rm in} \neq \emptyset$ and $\sigma_{\scriptscriptstyle H}(\Omega _{\rm out})=\Omega _{\rm out}$, then
			\begin{align}\label{f-k_pol2}
				\tau _{1}(P_H(\Omega_{\rm out} \setminus \overline{\Omega _{\rm in}})) \leq \tau _{1}(\Omega_{\rm out} \setminus \overline{\Omega _{\rm in}}).
			\end{align}
			\item If $\frac{2d+2}{d+2}<p<\infty$, and the equality holds in~\eqref{f-k_pol1}~or~\eqref{f-k_pol2} then  
			\begin{align*}
				P_H(\Omega_{\rm out} \setminus \overline{\Omega _{\rm in}})=\Omega_{\rm out} \setminus \overline{\Omega _{\rm in}} \mbox{ or }   P_H(\Omega_{\rm out} \setminus \overline{\Omega _{\rm in}})=\sigma _{\scriptscriptstyle H}(\Omega_{\rm out} \setminus \overline{\Omega _{\rm in}}).
			\end{align*}
		\end{enumerate}
	\end{theorem}
	\begin{remark}
		\noindent
		\begin{enumerate}[label=(\roman*)]
		    \item If $\Omega _{\rm in}=\emptyset ,$ in (i) then  $\nu _1 $ corresponds to the first  Dirichlet eigenvalue $\lambda _1$ and thus~\eqref{f-k_pol1} gives: for every $H\in \mathscr{H}_{\rm ad}$,
		\begin{align*}
		    \lambda _1(P_H(\Omega_{\rm out} \setminus \overline{\Omega _{\rm in}}))\leq \lambda _1(\Omega_{\rm out} \setminus \overline{\Omega _{\rm in}}).
		\end{align*}
		\vspace{1mm}
		\item If $\Omega _{\rm in}=\emptyset ,$ then $\tau _1(\Omega _{\rm out})= \tau _1 (P_H(\Omega _{\rm out}))=0,$ for every $H\in \mathscr{H}.$  This is the reason why we impose the condition $\Omega _{\rm in} \neq \emptyset$ in (ii). 
		\vspace{1mm}
		
		\item The symmetry assumptions in (i) and (ii) of Theorem~\ref{thm~1.2} ensure that $\Gamma _{\scriptscriptstyle N} \subseteq \partial P_H(\Omega )$ and hence the Neumann boundary is unaltered under such  polarizations. This fact is crucially used in our proof. Obtaining the same conclusions of Theorem~\ref{thm~1.2} without these additional symmetry assumptions seems to be a challenging problem.
		\end{enumerate}

			\end{remark}

	\subsection*{Application to the domain variations:}
	Next, we apply Theorem~\ref{thm~1.2} for the domains of the form $\Omega \setminus \mathscr{O} \subset \mathbb{R}^d$ to study the monotonicity of the first eigenvalue of~\eqref{eqn_intro_eigenZaremba1} on $\Omega \setminus \mathscr{O}$ under certain translations and rotations of $\mathscr{O}$ within $\Omega $. We assume the following:
	\begin{enumerate}[label={$\bf (A_{\arabic*})$}]
		\setcounter{enumi}{0}
	    \bigskip
	    \item \label{obstacle0}
	    $\mathscr{O} \subset \Omega $ is a closed set with nonempty interior such that $\Omega \setminus \mathscr{O}$ is a bounded Lipschitz domain.
	\end{enumerate}
	\bigskip 
	The set $\mathscr{O}$  in~\ref{obstacle0} is called as \emph{an obstacle}. The main idea is to express the translations and the rotations of $\mathscr{O}$ in terms of polarizations of punctured domain $\Omega \setminus \mathscr{O}$. Then we apply Theorem~\ref{thm~1.2}  and get the monotonicity of the eigenvalue. 
	
	\vspace{2mm}
	\noindent
	\textbf{The monotonicity along a straight line:} In this case, we set $\Omega _{\rm in}=\emptyset $ and $\Omega _{\rm out}=\Omega \setminus \mathscr{O}$ is a bounded Lipschitz domain in $\mathbb{R}^d$.  For a given $h\in \mathbb{S}^{d-1}$, we  study the monotonicity of the first Dirichlet eigenvalue of the $p$-Laplace operator with respect to the translations of the obstacle $\mathscr{O}$ in the $h$-direction within $\Omega $. Without loss of generality, we may assume that the origin $0\in \mathscr{O}$. We consider  the following family of polarizers:
	\begin{align}\label{fam-polarizers}
	    H_s =\bigl\{x\in \mathbb{R}^d: x\cdot h < s \bigr\}, \mbox{ for } s\in \mathbb{R}.
	\end{align}
	We make the following geometric assumption on $\Omega $ and $\mathscr{O}$:
	\vspace{1mm}
	\begin{enumerate}[label={$\bf (A_{\arabic*})$}]
		\setcounter{enumi}{1}
		\vspace{3mm}
		\item \label{ansatz1} $P_{H_0}(\Omega )=\Omega $, and $\mathscr{O}$ is Steiner symmetric with respect to the hyperplane $\partial H_{0}$ (see~Definition~\ref{defn-symm}). 
	\end{enumerate} 
	\vspace{3mm}
	The translations of $\mathscr{O}$ in the directions of $h$ are given by
	\begin{align}
	    \mathscr{O}_s= s h +\mathscr{O} \mbox{ for } s \in \mathbb{R}.
	\end{align}
	For $\Omega $ and $\mathscr{O}$ as given in~\ref{ansatz1}, define $\mathrm{L}_\mathscr{O}= \Big\{s \in \mathbb{R} : P_{\scriptscriptstyle H_s} (\Omega ) = \Omega \mbox{ and } \mathscr{O}_s \subset \Omega \Big\}.$ Let $\lambda _1(s)$ be the first eigenvalue of~\eqref{eigen1} with $\Omega _{\rm in}=\emptyset $ and $\Omega _{\rm out}= \Omega \setminus \mathscr{O}_s$ for $s\in \mathrm{L}_\mathscr{O}$. For $s \in \mathbb{R}$, let $\Sigma _s \mathrel{\mathop:}=\bigl\{x\in \Omega : x\cdot h >s \bigr\}$. A set $A\subseteq \mathbb{R}$ is said to be convex in the $h$-direction, if any line segment parallel to the $\mathbb{R}h$-axis with endpoints in $ A$ completely lies in $A$. Now, we have the following strict monotonicity result.
	\begin{theorem}\label{thm-mono}
		Let $\frac{2d+2}{d+2}<p<\infty $ and $h\in \mathbb{S}^{d-1}$. Assume that $\mathscr{O}, \Omega \subset \mathbb{R}^d$ satisfy \ref{obstacle0} and \ref{ansatz1}. If the set $\Sigma _{s_0} \bigcup \sigma _{\scriptscriptstyle H_{s_0}} \left(\Sigma _{s_0} \right)$ is convex in the $h$-direction for some $s_0 \in \mathrm{L}_\mathscr{O}$, then the set $\{s\in \mathrm{L}_\mathscr{O}:s\ge s_0\}$ is an interval and  $\lambda _1(\cdot )$ is strictly decreasing on this interval. 
	\end{theorem}
	Throughout this article, for given $a\in \mathbb{R}^d$ and $r\geq 0$, we denote $B_r(a)=\left\{x\in \mathbb{R}^d:|x-a|<r\right\}$, the open ball centered at $a$ with the radius $r$, and the closure of $B_r(a)$ by $\overline{B}_r(a)$.
	\begin{remark}\label{rmk-1}
		In Theorem~\ref{thm-mono}, if $\Omega $ itself is convex in the $h$-direction, then $\mathrm{L}_\mathscr{O}$ is an interval containing $0$. In particular, if $\Omega =B_R(0)$, $\mathscr{O}=\overline{B}_r(0)$ for $0<R<r<\infty $ and $h=e_1=(1,0,\dots ,0)\in \mathbb{S}^{d-1}$. Then, both $\Omega $ and $\mathscr{O}$ are Steiner symmetric with respect to $\partial H_0$, and $\mathrm{L}_\mathscr{O}=[0,R-r)$. Therefore, by Theorem~\ref{thm-mono}, the first Dirichlet eigenvalue $\lambda _1 (B_R(0) \setminus \overline{B}_r(s e_1))$ is strictly decreasing for $s\in [0,R-r)$. Thus, Theorem~\ref{thm-mono} gives an alternate proof for many existing strict monotonicity results that were proved using the shape derivative (Hadamard perturbation) formula. For example, Kesavan~\cite{Kesavan} and Harrell-Kr\"oger-Kurata~\cite{HarrelKroger}, and  Anoop-Bobkov-Sasi~\cite{Anoop} for $p\in \left(\frac{2d+2}{d+2},\infty \right)$. 
	\end{remark}
	\begin{remark}
	    Due to the  symmetry restrictions on the Neumann boundary in Theorem~\ref{thm~1.2}, the monotonicity results (similar to that of Dirichlet eigenvalue in Remark~\ref{rmk-1}), when the Neumann boundary condition is specified on $\partial B_R(0)$ can not be deduced from Theorem~\ref{thm-mono}. However, such a monotonicity result is proved for $p=2$, by Anoop-Ashok-Kesavan~\cite{Anoop-Ash-Kesh} using the Hadamard perturbation formula and some geometric properties of the first eigenfunctions. This result is open for general $p\neq 2$.
	\end{remark}
	
	\noindent
	\textbf{The monotonicity with respect to the rotations about a point:} Next, we study the monotonicity of the first eigenvalue of~\eqref{eqn_intro_eigenZaremba1} on $\Omega \setminus \mathscr{O}$ with respect to the rotations of the obstacle $\mathscr{O}$ in $\Omega $ about a point $a\in \mathbb{R}^d$. We set $\mathbb{R}^+=[0, \infty )$, and make the following geometric assumptions on both $\Omega $ and $\mathscr{O}$:
	\begin{enumerate}[label={$\bf (A_{\arabic*})$}]
		\setcounter{enumi}{2}
		\item \label{assumption1} The domain $\Omega $ and the obstacle $\mathscr{O}$ are foliated Schwarz symmetric with respect to the ray $a+\mathbb{R}^+ \eta $, for some $\eta \in \mathbb{S}^{d-1}$ (see~Definition~\ref{defn-symm}).
	\end{enumerate}
	For $s\in [-1,1],$ let $\theta _s:=\arccos (s)\in [0,\pi ]$. For $\xi \in \mathbb{S}^{d-1}\setminus \{\eta \}$, let $R_{s, \xi }$ be the simple rotation on $\mathbb{R}^d$ with the plane of rotation is $X_\xi := \mbox{span} \left\{\eta,\xi \right\}$ and the angle of rotation is $\theta _s$ from the ray $\mathbb{R}^+\eta $  in the counter-clockwise direction. The rotation of the obstacle $\mathscr{O}$ by $R_{s, \xi }$ about the point $a$ is given by 
	\begin{align}\label{obstacle}
		\mathscr{O}_{s,\xi }= a+R_{s,\xi }(-a+\mathscr{O}).
	\end{align}
	Now, we observe the following facts (see Proposition~\ref{fssObs} and Lemma~\ref{lemma:fssObs}):
	\begin{enumerate}[label=(\alph*)]
	    \item The rotated obstacle $\mathscr{O}_{s,\xi }$ is foliated Schwarz symmetric with respect to the ray $a+ \mathbb{R}^+ R_{s,\xi}(\eta )$.
	    \item For any rotation $R$ that fixes $\eta $,  $\mathscr{O}=a+R(-a+\mathscr{O})$ and $\Omega = a+R(-a+\Omega )$. 
	    \item For any distinct $\xi _1, \xi _2 \in \mathbb{S}^{d-1}\setminus \{\eta \},$
	there exists $R$  that fixes $\eta$ such that 
	\begin{align*}
	   R(-a+\Omega\setminus \mathscr{O}_{s,\xi _1})&=-a+\Omega \setminus \mathscr{O}_{s,\xi_2}.
	\end{align*}
	\end{enumerate}
	From the above observations, it is evident that we only need to consider the rotations of the obstacle by $R_{s,\xi}$ with respect to $a$ in a   $X_\xi$-plane for a fixed $\xi  \in \mathbb{S}^{d-1}\setminus \{\eta \}.$ Thus for $s\in [-1,1],$ we set $\mathscr{O}_s=\mathscr{O}_{s,\xi}$ and  consider
	 \begin{align}
	    \begin{aligned}
	        \mathrm{C}_\mathscr{O} &:= \Bigl\{s \in [-1,1]: \mathscr{O}_{s}  \subset \Omega \Bigr\},\\
	       \gamma_{1} (s )  &:= \gamma_{1}(\Omega \setminus \mathscr{O}_{s}), \text{ the first eigenvalue of \eqref{eqn_intro_eigenZaremba1} on } \Omega \setminus \mathscr{O}_{s} \text{ for } s\in \mathrm{C}_\mathscr{O}.
	    \end{aligned}
	 \end{align}
	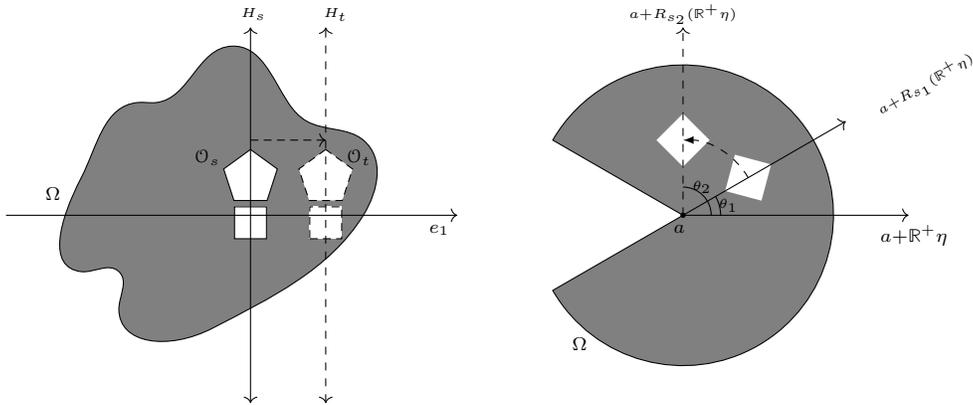
\begin{figure}[hbt!]
		\centering
		\begin{tikzpicture}[scale=0.5,use Hobby shortcut,closed=true]
			\draw [xshift=-1cm,fill=gray] (-3,1).. (-1.5,3).. (-1,3).. (1,4.5).. (2.5,3).. (3,2.5).. (4.5,2).. (4.5,0).. (0.5,-3).. (-2,-2.5).. (-2,-1.5).. (-3,-1.5).. (-3,1);
			\node [xshift=-0.5cm,fill=white!100,regular polygon, draw, regular polygon sides = 5, minimum size=0.75cm] (p) at (1.5,1) {};
			\node [xshift=-0.5cm,fill=white!100,regular polygon, draw, regular polygon sides = 4, minimum size=0.6cm] (p) at (1.5,-0.2) {};
			\node [xshift=-0.5cm,densely dashed,fill=white!100,regular polygon, draw, regular polygon sides = 5, minimum size=0.75cm] (p) at (3.5,1) {};
			\node [xshift=-0.5cm,densely dashed,fill=white!100,regular polygon, draw, regular polygon sides = 4, minimum size=0.6cm] (p) at (3.5,-0.2) {};
			\draw [xshift=-1cm,dashed,<->] (3.5,-5) -- (3.5,5);
			\draw [xshift=-1cm,<->] (1.5,-5) -- (1.5,5);
			\draw [xshift=-0.5cm,densely dashed,->] (1,2) -- (3,2);
			\draw [->] (-6,0) -- (6,0);
			\draw (-5.2,1) node[anchor =north west] {$\scriptstyle \Omega $};
			\draw (-1.25,2) node[anchor =north west] {$\scriptstyle \mathscr{O}_s$};
			\draw (2.8,2) node[anchor =north west] {$\scriptstyle \mathscr{O}_t$};
			\draw (5,0) node[anchor =north west] {$\scriptstyle e_1$};
			\draw (0,5.75) node[anchor = north west] {$\scriptscriptstyle H_s$};
			\draw (2.2,5.75) node[anchor =north west] {$\scriptscriptstyle H_t$};
			\begin{scope}[xshift=12cm]
				\draw [rotate around={210:(0,0)},fill=gray!100] (0,0) -- (4,0) arc[start angle=0, end angle=300,radius=4cm] -- (0,0);
				\node [rotate=75,regular polygon sides = 4, minimum size=0.5cm, fill=white, thick] (p) at (2*cos 30,2*sin 30) {};
				\node [rotate=135,regular polygon sides = 4, minimum size=0.5cm,fill=white, densely dotted] (p) at (2*cos 90,2*sin 90) {};
				\fill (0,0) circle[radius=2pt];
				\draw [dashed,-latex] (30:2) arc (30:90:2);
				\draw (0:1) arc (0:30:1);
				\draw (20:0.75) node[anchor = west ] {$\scriptscriptstyle \theta 	_1$};
				\draw (0:0.75) arc (0:90:0.75);
				\draw (90:0.75) node[anchor = west] {$\scriptscriptstyle \theta _2$};
				\draw [->] (0,0)--(6,0);
				\draw [dashed,->,rotate=90] (0,0)--(5,0);
				\draw [->,rotate=30] (0,0)--(5,0);
				\draw [rotate=90,dotted] (4,0) node[anchor =north, label={[label distance=0.5cm,text depth=-1ex,rotate=0] $\scriptscriptstyle a+R_{s_2} (\mathbb{R}^+ \eta )$}] {};
				\draw [rotate=18] (7.5,0) node[anchor =north, label={[label distance=0.5cm,text depth=-1ex,rotate=30] $\scriptscriptstyle a+ R_{s_1}(\mathbb{R}^+ \eta )$}] {};
				\draw (5,0) node[anchor =north west] {$\scriptstyle a+\mathbb{R}^+\eta $};
				\draw (-3.2,-3) node[anchor =north west] {$\scriptstyle \Omega $};
				\draw (-0.5,0) node[anchor =north west] {$\scriptstyle a$};
			\end{scope}
		\end{tikzpicture}
		\caption{The translations of $\mathscr{O}$ along the $e_1$-axis; and rotations of $\mathscr{O}$ about the point $a\in \mathbb{R}^d$, here $\theta _i =\arccos\left(s_i\right)$ for $i=1,2$ with $s_1>s_2$.}
	\end{figure}
	In this article, we consider the following  types of $\Omega $ and $\Gamma _{\scriptscriptstyle N} \subseteq \partial \Omega $: 
	\vspace{1mm}
	\begin{enumerate}[label={$\bf (A_{\arabic*})$}]
		\setcounter{enumi}{3}
		\item \label{choice-1} $\Omega =\Omega _0 \setminus \overline{B}_{\rho _0} (a)$, where $\overline{B}_{\rho _0}(a)\subsetneq \Omega _0 \subset \mathbb{R}^d$, $\rho _0\geq 0$; and $\Gamma _{\scriptscriptstyle N} = \partial B_{\rho _0}(a)$. 
		\vspace{1mm}
		\item \label{choice-2} $\Omega = B_R(a)\setminus \overline{\Omega _1}$, where $\overline{\Omega _1}\subset B_R(a)$, and $\Gamma _{\scriptscriptstyle N} = \partial B_R(a)$. 
	\end{enumerate}
	Now, we state  our monotonicity result for $\gamma_1(.)$ on $\mathrm{C}_\mathscr{O}$. 
	\begin{theorem}\label{thm-circ}
		Let $\frac{2d+2}{d+2}<p<\infty $ and $\Omega \subset \mathbb{R}^d$ be a domain. Assume that the pair $\Omega $ and $\Gamma _{\scriptscriptstyle N}$ satisfy either~\ref{choice-1} or~\ref{choice-2}. If $\Omega $ and $\mathscr{O}$  satisfy~\ref{obstacle0} and~\ref{assumption1} for some $a\in \mathbb{R}^d$ and $\eta\in \mathbb{S}^{d-1}$, then $\mathrm{C}_\mathscr{O}$ is an interval. In addition, if $\Omega$ is not radial with respect to $a$, then $\gamma _{1} (\cdot )$ is strictly increasing on $\mathrm{C}_\mathscr{O}$.
	\end{theorem}
	\begin{remark}
		If  $\Omega $ is radial with respect to the point $a$ (see Corollary~\ref{coro-radialsymm}), then the first eigenvalue $\gamma _{1}(\cdot )$ remains as a constant on $\mathrm{C}_\mathscr{O}$. 
	\end{remark}
	The rest of this article is organized as follows. In Section~\ref{Prelim}, the polarization of measurable sets and functions are introduced, and some of their important properties are discussed. Further, the characterizations of Steiner and foliated Schwarz symmetries using polarizations are given in Section~\ref{Prelim}. Also, we include a strong comparison principle and a few interior and boundary regularity results that are essential for the development of this article. A proof of Faber-Krahn inequality (Theorem~\ref{thm~1.2}) is given in Section~\ref{FK type}. The proofs of strict monotonicity results (Theorem~\ref{thm-mono} and~\ref{thm-circ}) are given in Section~\ref{MonoEigenvalue}. Many important remarks and explicit examples are included in Section~\ref{Remarks}.
	\section{Preliminaries}\label{Prelim}
	In this section, we discuss some of the important properties of the polarization of the sets and functions. Further, we give the definitions of Steiner and foliated Schwarz symmetries, and their characterizations in terms of polarizations. Lastly, we give some regularity results and strong comparison principles for the solutions of the $p$-Laplace operator. 
	\subsection{Polarization of sets}
	We discuss a few simple properties of the polarization of sets.
	\begin{proposition}\label{Propo-Obs}
		Let $H\in \mathscr{H}$ and $A, C \subseteq \mathbb{R}^d$. Then,
		\begin{enumerate}[label={\rm (\roman*)}]
			\item $P_H(A)$ is open, if $A$ is open; and $P_H(A)$ is closed if $A$ is closed;
			\item $P_H(A)\subseteq P_H(C)$, if $A\subseteq C$;
			\item $P_H(A\cap C)\subseteq P_H(A)\cap P_H(C)$ and $P_H(A) \cup P_H(C) \subseteq P_H(A\cup C)$;
			\item $P_H(\sigma _H(A))=P_H(A)$, $\sigma _H(P_H(A))={P}^H(A),$ and $\sigma _H({P}^H(A))=P_H(A)$;
			\item $P_H(A^\mathsf{c})=\left(P^H(A)\right)^\mathsf{c}$. 
		\end{enumerate}
	\end{proposition}
	\begin{proof}
		Recall that, for $H\in \mathscr{H}$, the polarizations of a set $A\subseteq \mathbb{R}^d$ are given by
		\begin{align*}
		    P_H(A)=\left[\left(A\cup \sigma _H(A)\right)\cap H\right]\cup \left[A\cap \sigma _H (A)\right], \mbox{ and }
		    P^H(A)=\left[\left(A\cup \sigma _H(A)\right)\cap H^\mathsf{c} \right] \cup \left[A\cap \sigma _H(A)\right].
	    \end{align*}
		Since $A\cap \partial H=\sigma _H(A)\cap \partial H$, we can also write
		\begin{align*}
		    P_H(A)=\left[\left(A\cup \sigma _H(A)\right)\cap \overline{H}\right]\cup \left[A\cap \sigma _H (A)\right] \mbox{ and } P^H(A)=\left[\left(A\cup \sigma _H(A)\right)\cap \overline{H}^\mathsf{c}\right]\cup \left[A\cap \sigma _H (A)\right].
		\end{align*}
		Now, (i)-(iii) follow easily from the above observations.\\ 
		(iv) This follows from the fact $\sigma _H(H)=\overline{H}^\mathsf{c}$ and the above observations.\\
		(v) By the definition, we have $ P_H(A^\mathsf{c})= \left[\left(A^\mathsf{c}\cup \sigma_H(A^\mathsf{c})\right)\cap H \right]\bigcup \left[A^\mathsf{c} \cap \sigma_H(A^\mathsf{c})\right],$ and hence
		\begin{align*}
			\left(P_H(A^\mathsf{c})\right)^\mathsf{c}	&= \left[\left(A \cap \sigma_H(A)\right)\cup H^\mathsf{c}\right] \cap \left[A\cup \sigma_H(A)\right] \\
			&=\left[A\cap \sigma_H(A)\right] \cup \left[\left(A\cup \sigma_H(A) \right) \cap H^\mathsf{c}\right] = P^H(A). \qedhere
		\end{align*}	
	\end{proof}
	The following proposition characterizes the invariance of a set under polarizations.  
	\begin{proposition}\label{propo:pol-equiv}
	    Let $H\in \mathscr{H}$ and $A\subseteq \mathbb{R}^d$. Then 
	    \begin{enumerate}[label=\rm (\roman*)]
	        \item $P_H (A)=A$ if and only if $\sigma _H(A) \cap H \subseteq A;$
	        \item $P^H (A)=A$ if and only if $\sigma _H(A) \cap H^c \subseteq A;$
	        \item $P_H (A)=P^H(A)$ if and only if $\sigma _H(A)=A.$
	    \end{enumerate}
	\end{proposition}
	\begin{proof}
	    (i) From the definition of $P_H(A)$, it is clear that 
		\begin{equation}\label{def:PH3}
			P_H(A)=\left[\left(A\cup \sigma _H(A)\right)\cap H\right]\cup \left[A\cap \sigma _H (A)\cap {H}^\mathsf{c}\right].
		\end{equation}
		If $P_H(A)=A$, then $P_H(A)\cap H=A\cap H.$ Thus the above equation yields $\left(A\cup \sigma _H(A)\right)\cap H \subseteq A$, and hence we must have $\sigma _H(A) \cap H\subseteq A.$
		Conversely, assume that $\sigma _H(A)\cap H\subseteq A.$  Then, by applying $\sigma_H$ on both sides, and using the fact that $A \cap \partial H= \sigma_H(A) \cap \partial H,$ we obtain  $A\cap H^\mathsf{c} \subseteq \sigma _H(A),$   Therefore, $A\cap {H}^\mathsf{c}=A\cap \sigma_H (A)\cap {H}^\mathsf{c}.$ 
		From the assumption, we also have $A \cap{H}=\left[\left(A\cup  \sigma _H(A)\right)\cap {H} \right]$.  Now, using~\eqref{def:PH3}, we easily conclude  that $P_H(A)=A$.
		
		\noindent (ii) From Proposition~\ref{Propo-Obs}-(iv), we have $\sigma _H(P^H(A))=P_H(A)$ and $P_H(\sigma _H(A))=P_H(A)$. Therefore, we get $P^H(A)=A$ if and only if $P_H(\sigma _H(A))=\sigma _H(A).$ Now, from (i) we obtain
		\begin{align*}
		    P_H(\sigma _H(A))=\sigma _H(A) \mbox{ if and only if } A\cap H \subseteq \sigma _H(A).
		\end{align*}
		Now applying $\sigma _H$ on both sides of last inclusion and using the fact that $\sigma _H(A)\cap \partial H =A\cap \partial H$, we get
		\begin{align*}
		    P^H(A)=A \mbox{ if and only if } \sigma_H(A)\cap H^c \subseteq A.
		\end{align*}
		
		\noindent 
		(iii) From the definitions of $P_H(A)$ and $P^H(A)$, it is clear that 
		\begin{align*}
			P_H(A)=\left[\left(A\cup \sigma _H(A)\right)\cap H\right]\cup \left[A\cap \sigma _H (A)\cap {H}^\mathsf{c}\right],\\
			P^H(A)=\left[\left(A\cup \sigma _H(A)\right)\cap H^\mathsf{c}\right]\cup \left[A\cap \sigma _H (A)\cap H\right].
		\end{align*}
		If $P_H(A)=P^H(A)$, then $\left(A\cup \sigma _H(A)\right)\cap H = A\cap \sigma _H (A)\cap H$ and $A\cap \sigma _H (A)\cap {H}^\mathsf{c}=\left(A\cup \sigma _H(A)\right)\cap H^\mathsf{c}$. Therefore $A\cup \sigma _H(A)=A\cap \sigma _H(A)$, and hence $\sigma _H(A)=A$. Conversely, assume that $\sigma _H(A)=A$. Then, from above equations, we get $P_H(A)=P^H(A)=A$.
	\end{proof}
	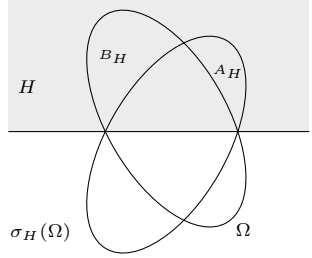
\begin{figure}[hbt!]
		\centering
		\begin{tikzpicture}[scale=0.2]
			\fill[gray!15] (-10, 0) rectangle (10,9);
			\draw[rotate around = {60:(1,0)}] (0, 0) ellipse (8cm and 4cm);
			\draw[rotate around = {-60:(1,0)}] (0, 0) ellipse (8cm and 4cm);
			\draw (-10,0) -- (10,0);
			\draw (-10,4) node[anchor=north west] {$\scriptstyle H$};
			\draw (-10.5,-5.5) node[anchor=north west] {$\scriptstyle \sigma _H(\Omega )$};
			\draw (4.5,-5.5) node[anchor=north west] {$\scriptstyle \Omega $};
			\draw (-3,5) node {$\scriptscriptstyle B_H$};
			\draw (4.6,4) node {$\scriptscriptstyle A_H$};
		\end{tikzpicture}
		\caption{The sets $A_H$ and $B_H$ of $P_H(\Omega )\cap H$.}
	\end{figure}
	\begin{proposition}\label{propo:pol-equiv1}
		Let $H\in \mathscr{H}$ and $\Omega \subseteq \mathbb{R}^d$ be an open set. Then,
		\begin{enumerate}[label={\rm (\roman*)}]
			\item $P_H(\Omega ) \neq \Omega $ if and only if $A_H \mathrel{\mathop:}= \sigma _H (\Omega ) \cap \Omega ^\mathsf{c} \cap H$ has non-empty interior;
			\item $P_H(\Omega ) \neq \Omega $ if and only if $B_H \mathrel{\mathop:}= \Omega \cap \sigma _H (\Omega ^\mathsf{c}) \cap H$ has non-empty interior.
		\end{enumerate}
	\end{proposition}
	\begin{proof}
		(i) First, we observe that the interior of $A_H$ is $\sigma _H (\Omega ) \cap \overline{\Omega }^\mathsf{c} \cap H$. Since $\sigma _H (\Omega ) \cap H$ is open, from  Proposition~\ref{propo:pol-equiv}, we get    $P_H(\Omega ) \neq \Omega $ if and only if $\sigma _H (\Omega ) \cap H \nsubseteq \overline{\Omega }$. Clearly, $\sigma _H (\Omega ) \cap H \nsubseteq \overline{\Omega }$ if and only if $\overline{\Omega }^\mathsf{c} \cap \sigma _H (\Omega ) \cap H \ne \emptyset $.\\
		(ii) For $H\in \mathscr{H}$, we have $\sigma _H(H) \in \mathscr{H}$. Then from Proposition~\ref{Propo-Obs}, $P_H(\Omega )=\sigma _H (\Omega )$ if and only if $P_{\sigma _H(H)} (\Omega )=\Omega $. The proof follows from (i) by replacing $H$ with $\sigma _H(H)$.
	\end{proof}
	\begin{remark}
		For an open set $\Omega \subset \mathbb{R}^d$, if $P_H(\Omega ) \neq \Omega $ then the interior of $P_H(\Omega ) \setminus \Omega $ is non-empty. Therefore, if $P_H(\Omega ) \neq \Omega $ then $P_H(\Omega )$ can not be equal to $\Omega $ up to a set of measure zero (or up to a set of $p$-capacity zero).
	\end{remark}
	Now, we prove that the set $P_H(\Omega )\cap H$ is a domain when $\Omega $ is a domain. For this, we need the following lemma.
	\begin{lemma}\label{lemma-symm_conn}
		Let $H\in \mathcal{H}$ and $\Omega \subseteq \mathbb{R}^d$ be a domain. If $\sigma _H(\Omega )=\Omega $, then both $\Omega \cap H$ and $\Omega \cap \overline{H}^\mathsf{c}$ are connected.
	\end{lemma}
	\begin{proof}
		Let $f:\Omega \cap \overline{H}\longrightarrow \{0,1\}$ be a continuous function. Using the symmetry of $\Omega $, we define 
		\begin{align*}
			\widetilde{f}(x) = \left\{
			\begin{aligned}
				&f(x),  & \mbox{for } x\in \Omega \cap \overline{H},\\
				&f\circ \sigma _H (x),  & \mbox{for } x\in \Omega \cap H^\mathsf{c}.
			\end{aligned}
			\right.
		\end{align*}
		Then $\widetilde{f}$ is a continuous function on $\Omega $, since $\sigma _H(x)=x$ for $x\in \partial H$. By the connectedness of $\Omega $, $\widetilde{f}$ is constant on $\Omega $. In particular, $f$ is constant on $\Omega \cap \overline{H}$, and hence $\Omega \cap \overline{H}$ is connected. Therefore ${\rm int}\left(\Omega \cap \overline{H}\right)=\Omega \cap {H}$ is connected, and hence $\sigma_H(\Omega \cap H)=\Omega \cap \overline{H}^\mathsf{c}$ is also connected.
	\end{proof}
	\begin{proposition}\label{propo-PHdomain}
		Let $H\in \mathscr{H} $ and $\Omega \subseteq \mathbb{R}^d$ be a domain. Then $P_H(\Omega )\cap H$ is a domain.
	\end{proposition}
	\begin{proof}
		First, we observe that $P_H(\Omega )\cap H=\left(\Omega \cup \sigma _H(\Omega ) \right)\cap H$ is open. For proving the connectedness, we consider the following two cases: \textbf{(a)} $\Omega \cap \sigma _H(\Omega)=\emptyset $, and \textbf{(b)} $\Omega \cap \sigma _H(\Omega ) \neq \emptyset $.
		
		\vspace{3mm}
		\noindent \underline{\textbf{(a)} $\Omega \cap \sigma _H(\Omega)=\emptyset$:} In this case, we have $\Omega \cap \partial H =\emptyset $, since  $\Omega \cap \partial H\subset \Omega \cap \sigma _H(\Omega)$. Therefore, $\Omega $ is the union of two open sets $\Omega \cap H$ and $\Omega \cap {\overline{H}^\mathsf{c}}$. By the connectedness of $\Omega $, one of them is equal to $\Omega $. If $\Omega \cap H=\Omega $ then $P_H(\Omega ) =\Omega $, and hence $P_H(\Omega )\cap H=\Omega $. If $\Omega \cap \overline{H}^\mathsf{c} =\Omega $ then $P_H(\Omega ) = \sigma _H(\Omega )$, and hence $P_H(\Omega ) \cap H=\sigma _H(\Omega )$. 
		
		\vspace{3mm}
		\noindent \underline{\textbf{(b)} $\Omega \cap \sigma _H(\Omega )\neq \emptyset $:} In this case, we have $\Omega \cup \sigma _H(\Omega )$ is connected, and it is symmetric with respect to $\partial H$. Thus, by Lemma~\ref{lemma-symm_conn}, $\left(\Omega \cup \sigma _H(\Omega ) \right) \cap H=P_H(\Omega )\cap H$ is connected.
		
		\noindent
		Therefore, in both of the cases, $P_H(\Omega )\cap H$ is domain.
	\end{proof}
	\noindent \textbf{The Steiner, axial and foliated Schwarz symmetries:} A set in $\mathbb{R}^d$ is said to have certain symmetry, if it is invariant under corresponding symmetrization or rearrangement on $\mathbb{R}^d$.  Here, we directly give the definitions of  the Steiner and the foliated Schwarz symmetries without  defining the associated symmetrizations (see~\cite[Definition~3.1 and Definition~3.2]{VanJeanWillem08}). The foliated Schwarz symmetrization with respect to a ray $a+\mathbb{R}^+\eta $ is the cap symmetrization with respect to $a+\mathbb{R}^+\eta $, see~\cite[Definition~3.2]{VanJeanWillem08}. 
	\begin{definition}\label{defn-symm}
	Let $A\subseteq \mathbb{R}^d$ be a measurable set.
	\begin{enumerate}[align=right,label=\rm (\arabic*)]
		\item\label{def-Steiner}\textbf{Steiner symmetry.}
		Let $S$ be an affine-hyperplane in $\mathbb{R}^d$. For each $x\in S,$ let $L_x$ be the line passing through $x$ and orthogonal to $S.$ Then $A$ is said to be Steiner symmetric with respect to $S$, if 
		\begin{align*}
		    \text{  for each } x\in S,\quad A\cap L_x= B_\rho(x)\cap L_x \text{ for some } \rho\ge 0.
		\end{align*}
		
		\item \label{def-axial}\textbf{Axial symmetry.} 
		Let $L$ be a  line in $\mathbb{R}^d$. For each $x\in L,$ let $S_x$ be the affine hyperplane passing through $x$ and orthogonal  to $L$. Then $A$ is said to be axially symmetric with respect to  $L,$ if 
		\begin{align*}
		    \text{  for each } x\in L,\quad A\cap \partial B_\rho (x)\cap S_x= \partial B_\rho (x)\cap S_x \text{ for some } \rho\ge 0.
		\end{align*}
		
		\item\label{def-fss}\textbf{Foliated Schwarz symmetry.}
		Let $a+\mathbb{R}^+\eta $ be a ray starting for some $a\in \mathbb{R}^d$ and $\eta \in \mathbb{S}^{d-1}$. Then $A$ said to be foliated Schwarz symmetric with respect to $a+\mathbb{R}^+\eta $, if 
		\begin{align*}
		    \text{for every } r>0,\quad A\cap \partial B_r(a)= B_\rho (a+r \eta ) \cap \partial B_r(a) \text{ for some } \rho \geq 0. 
		\end{align*} 
	\end{enumerate}
	\end{definition}
	\begin{remark}\label{Observations}
	    We observe that:
	    \begin{enumerate}[label=(\roman*)]
	        \item a set $A\subseteq \mathbb{R}^d$ is Steiner symmetric with respect to an affine-hyperplane $S,$ if and only if  $A$ is invariant under the reflection with respect to $S$ and convex in the orthogonal direction to  $S$;
	        \item a set $A\subseteq \mathbb{R}^d$ is axially symmetric with respect to a line $L,$  if and only if $A$ is invariant under the reflection with respect to every affine hyperplane containing L. In particular, if $L=\mathbb{R}\eta$ and $R$ is any rotation on $\mathbb{R}^d$ such that $R(\eta)=\eta$, then $R(A)=A.$ This follows from the definition, since the planes of rotation of such $R$ can not contain $\eta$, and hence  those planes must be orthogonal to $\eta $; 
	        \item let $A\subseteq \mathbb{R}^d $ be  foliated Schwarz symmetric  with respect to  $a+\mathbb{R}^+ \eta.$ Let $I_A:=\bigl\{r>0 : A\cap \partial B_r(a)\ne \emptyset \bigr\}$. For $r\in I_A,$ let $\rho(r)>0$ be such that $A\cap \partial B_r(a)=B_{\rho(r)}(a+r\eta )\cap \partial B_r(a)$. Then, 
	        \begin{align}\label{eqn:fssSet}
	            A=\bigcup_{r\in I_A}  B_{\rho(r)}(a+r\eta )\cap \partial B_r(a).
	        \end{align}
	    \end{enumerate}
	\end{remark}
	The following proposition provides some  properties of the foliated Schwartz symmetric sets.
	\begin{proposition}\label{fssObs}
	   If $A\subseteq \mathbb{R}^d$ is foliated Schwarz symmetric with respect to a ray $a+\mathbb{R}^+ \eta $ then
	   \begin{enumerate}[label={\rm (\roman*)}]
	      \item $A$ is axially symmetric with respect to $a+\mathbb{R} \eta,$ 
	      \item for any linear map $T$  and  $b\in \mathbb{R}^d$, the set $b+T(A)$ is foliated Schwarz symmetric with respect to $b+T(a)+\mathbb{R}^+T(\eta )$,
	      \item $R(-a+A)=-a+A,$ for any rotation $R$ on $\mathbb{R}^d$ that fixes  $\eta $. 
	   \end{enumerate}
	\end{proposition}
	\begin{proof}
	   
	   (i) Observe that, for every $r\in I_A$, the set $B_{\rho (r)}(a+r \eta ) \cap \partial B_r(a)$ is axially symmetric with respect to $a+\mathbb{R}\eta.$ Now, using  ~\eqref{eqn:fssSet}, we conclude that $A$ is  axially symmetric  with respect to $a+\mathbb{R}\eta $.
	   
	   \noindent
	   (ii) Let  $r>0,$ then 
	   \begin{align*}
	    (b+T(A))\cap \partial B_r(b+T(a))&=b+T(A\cap \partial B_r(a)) =b+T(B_\rho (a+r\eta ) \cap \partial B_r(a)), \text{ for some } \rho \ge 0,
	    \end{align*}
	    where the last equality follows from  the definition foliated Schwarz symmetry. Thus 
	    \begin{align*}
	    (b+T(A))\cap \partial B_r(b+T(a))&=b+B_\rho (T(a)+rT(\eta) ) \cap \partial B_r(T(a))\\
	    &=B_\rho (b+T(a)+rT(\eta) ) \cap \partial B_r(b+T(a)).
	\end{align*}
	Now, we obtain the required conclusion by the definition of foliated Schwarz symmetry.
	
	\noindent
	(iii) By taking $T=I$ and $b=-a$ in (ii), we get  $-a+A$ is foliated Schwarz symmetric with respect to $\mathbb{R}^+\eta $. Thus by (i), $-a+A$ is axially symmetric with respect to   $\mathbb{R}\eta.$ Since $R$ fixes $\eta$, from (ii) of Remark \ref{Observations}, we conclude $R(-a+A)=-a+A.$
	\end{proof}
	Next, we characterize the foliated Schwarz and Steiner symmetric sets using the polarizations. First, we consider the following polarizers: for given  $a\in \mathbb{R}^d, \eta \in \mathbb{S}^{d-1},$ let
	\begin{align*}
	\mathscr{H}_{a,\eta }\mathrel{\mathop:}=\Big\{H\in \mathscr{H} : a+\mathbb{R}^+\eta \subset H \mbox{ and } a\in \partial H \Big\}.
	\end{align*}
	Some useful characterizations of the Steiner symmetry (from~\cite[Lemma~2.2]{Bobkov-Sergey19}), foliated Schwarz symmetry (from~\cite[Section~3]{VanJeanWillem08}) are given in the following proposition. 
	\begin{proposition}\label{propo:char}
	Let $A \subseteq \mathbb{R}^d$ be any set. 
	\noindent 
	\begin{enumerate}[label=\rm (\roman*)]
		\item \label{propo:equi-Steiner1} Let $H_s\in \mathscr{H}$ be as given in~\eqref{fam-polarizers}. Then, for $s_0\in \mathbb{R}$, the following statements are equivalent:
		\begin{enumerate}[label=\rm (\alph*)]
			\item the set $A$ is Steiner symmetric with respect to the affine-hyperplane $\partial H_{s_0}$,
			\item $P_{H_s} (A) = A$, for every $s \geq s_0$;  and $P^{H_s} (A) = A$, for every $s \leq s_0$.
		\end{enumerate}
		\item \label{char:fss} Let $a\in \mathbb{R}^d$ and $\eta \in \mathbb{S}^{d-1}$. Then the following are equivalent:
		\begin{enumerate}[label=\rm (\alph*)]
			\item the set $A$ is foliated Schwarz symmetric with respect to the ray $a+\mathbb{R}^+\eta $,
			\item $P^H(A)=\sigma _H(A)$, for every $H\in \mathscr{H}_{a,\eta }$.
		\end{enumerate}
	\end{enumerate}
	\end{proposition}
	We have the following corollary.
	\begin{corollary}\label{coro-radialsymm}
	   Let $A\subseteq \mathbb{R}^d$ be any set. If $A$ is foliated Schwarz symmetric with respect to both the rays $a+\mathbb{R}^+\eta $ and $a-\mathbb{R}^+ \eta $ for some $a\in \mathbb{R}^d$ and $\eta \in \mathbb{S}^{d-1}$. Then $A$ is radial with respect to the point $a$.
	\end{corollary}
	\begin{proof}
	   Notice that, $A$ is radial with respect to $a\in \mathbb{R}^d$ provided $A\cap \partial B_r(a) = \partial B_r(a)$ for every $r\in I_A$, where $I_A=\{r\in \mathbb{R}:A\cap \partial B_r(a)\neq \emptyset \}$. Since $A$ is foliated Schwarz symmetric with respect to both the rays $a+\mathbb{R}^+\eta $ and $a-\mathbb{R}^+ \eta $, for each $r\in I_A$ we get:
	   \begin{align}\label{eqn:coro-radialsymm}
	       A\cap \partial B_r(a)= B_{\rho _1}(a+r\eta ) \cap \partial B_r(a) = B_{\rho _2}(a-r\eta )\cap \partial B_r(a) \mbox{ for some } \rho _1, \rho _2\geq 0.
	   \end{align}
	   Since $|a-(a-r\eta )|=r$, from~\eqref{eqn:coro-radialsymm} we obtain $a-r\eta \in B_{\rho _1}(a+r\eta )$. Thus $\rho _1 \geq |a-r\eta -(a+r\eta )|=2r$, and hence $B_{\rho _1}(a+r\eta ) \cap \partial B_r(a) =\partial B_r(a)$. Now, from~\eqref{eqn:coro-radialsymm} we conclude that $A\cap \partial B_r(a)=\partial B_r(a)$.
	\end{proof}
	\subsection{Polarization of punctured domains} 
	We consider the polarization of the punctured domains of the form $A\setminus C$, where $A\subseteq \mathbb{R}^d$ is open, and $C\subset A$ is closed. Clearly $\partial (A \setminus C) = \partial A \sqcup \partial C$.
	\begin{proposition}\label{propo:pol-compliment1}
	   Let $A\subseteq \mathbb{R}^d$ be open and $C \subset A$ be closed. If $H \in \mathscr{H}$ is such that $\sigma _H(C) \subset A$, then
	   \begin{enumerate}[label=\rm (\roman*)]
		   \item $P_H(A \setminus C)= P_H(A)\setminus P^H(C)$,
		   \item $P^H(C) \subset P_H(A )$, in particular $\partial P_H(A \setminus C) = \partial P_H(A) \sqcup \partial P^H(C).$ 
	   \end{enumerate}
	\end{proposition}
	\begin{proof}
	   (i) For $A \subseteq \mathbb{R}^d$, denote $P_H^+(A)= P_H(A)\cap H$  and $P_H^-(A)=P_H(A)\cap H^\mathsf{c}$. Thus $P_H(A)=P_H^+(A)\sqcup P_H^-(A)$, and
	   \begin{align}\label{eqn:3.7}
		P_H(A)\cap P_H(C^\mathsf{c})= \left[P_H^+(A)\cap P_H^+(C^\mathsf{c}) \right] \sqcup \left[P_H^-(A)\cap P_H^-(C^\mathsf{c})\right].
	   \end{align}
	   On the other hand, we have $P_H^-(A\cap C^\mathsf{c}) = P_H^-(A) \cap P_H^-(C^\mathsf{c}).$	Since $C, \sigma _H(C) \subseteq A $, we get $\sigma _H(A )\cup C^\mathsf{c} = A \cup \sigma _H(C^\mathsf{c}) = \mathbb{R}^d$. Thus, $(A\cap C^\mathsf{c})\cup \sigma _H(A\cap C^\mathsf{c}) = \left(A \cup \sigma _H(A)\right) \cap \left(C^\mathsf{c}\cup \sigma _H (C^\mathsf{c}) \right)$, and hence $P_H^+(A \cap C^\mathsf{c})=P_H^+(A ) \cap P_H^+(C^\mathsf{c}).$ Therefore, from~\eqref{eqn:3.7} and using $P_H(C^\mathsf{c})=(P^H(C))^\mathsf{c}$ (Proposition~\ref{Propo-Obs}-(v)), we obtain 
	   \begin{align*}
		P_H(A \setminus C)=P_H(A ) \cap P_H(C^\mathsf{c})=P_H(A )\setminus P^H(C).
	   \end{align*}
	   (ii) Since $C\cup \sigma _H(C)$ is a symmetric set in $A,$ by the definitions of $P^H$ and $P_H$, we get $$P^H(C)\subseteq C\cup \sigma _H(C)=P_H(C\cup \sigma _H(C))\subset P_H(A).$$ Moreover,  $P^H(C)$ is closed and $P_H(A )$ is open in $\mathbb{R}^d.$ Thus,
	   \begin{align*}
 		\partial P_H(A\setminus C)&=\partial\left (P_H(A )\setminus P^H(C)\right) =\partial P_H(A) \sqcup \partial P^H(C). \qedhere 
	   \end{align*}
	\end{proof}
	\begin{remark}
	    The assumption $\sigma _H(C)\subset A$ is essential  for the conclusions of the above proposition. To see this, consider $A=B_R(0)$, $C=\overline{B}_r(0)$, and the polarizers $H_s:=\{x\in \mathbb{R}^d : x_1<s\}$ for $s\in \mathbb{R}$. For $s>\frac{R-r}{2}$, we have $\lvert\sigma _{H_t}(C)\cap A^\mathsf{c}\rvert=\lvert\overline{B}_r(2t e_1) \cap B_R(0)^\mathsf{c}\rvert>0$, where $|A|$ is the Lebesgue measure of $A\subseteq \mathbb{R}^d$. Then, $\lvert P_{H_t}(A)\setminus P^{H_t}(C)\rvert =\lvert B_R(0)\setminus \overline{B}_r(2t e_1)\rvert> \lvert B_R(0)\setminus \overline{B}_r(0)\rvert$. Since $P_H$ is measure preserving, we conclude that $P_H(A)\setminus P^H(C)\neq P_H(A\setminus C)$.
	\end{remark}
	\subsection{Polarization of functions} Now, we consider the	polarization of functions defined on a domain $\Omega \subseteq \mathbb{R}^d$ and discuss some important properties of polarization of functions, such as Lipschitz continuity, non-expansivity, norm preserving property. Recall the definition of polarization of functions (from Definition~\ref{defn-Polarization}). 
    \begin{proposition}\label{propo_support}
	   Let $H\in \mathscr{H}$, and $u\in \mathcal{C}(\mathbb{R}^d)$ be  a non-negative function. Then 
	   \begin{align*}
		{\rm supp} \left(P_H(u)\right) = P_H({\rm supp}\left(u\right)).
	   \end{align*}
    \end{proposition}
    \begin{proof} 
	   Let $F={\rm supp}\left(u\right)$. Clearly $u=u\circ \sigma _H=0$ on $F^\mathsf{c}\cap \sigma_H (F^\mathsf{c})$, and $u=0$ or $u\circ \sigma _H=0$ on $F^\mathsf{c} \cup \sigma _H(F^\mathsf{c})$. Since $u\geq 0$, by the definition, we get $P_H(u)=0$ on $\big[\left(F^\mathsf{c} \cup \sigma _H (F^\mathsf{c})\right)\cap H^\mathsf{c}\big]\cup \big[F^\mathsf{c}\cap \sigma_H(F^\mathsf{c}) \big] = P^H(F^\mathsf{c})$. Now, since $P^H(F^\mathsf{c}) = \left(P_H(F)\right)^\mathsf{c}$ (from~Proposition~\ref{Propo-Obs}-(v)), we get ${\rm supp}\left(P_H(u)\right)\subseteq P_H(F).$ The other way inclusion is easy to see from the definition. Therefore, ${\rm supp} \left(P_H(u)\right) = P_H({\rm supp}\left(u\right))$.
	\end{proof}
	\begin{remark}\label{rmk-1.7}
	   Similarly, for non-positive function $u\in \mathcal{C}(\mathbb{R}^d)$, ${\rm supp}\left(P_H(u)\right)\subseteq P^H({\rm supp}\left(u\right)).$ More generally, for any function $u\in \mathcal{C}(\mathbb{R}^d)$ we have ${\rm supp}\left(P_H(u)\right)=P_H({\rm supp}\left(u^+\right)) \cup P^H({\rm supp}\left(u^-\right))$ (see~\cite[Section-2]{Bobkov-Sergey19}), where $u^+=\max\{0,u\}$ and $u^-=\min\{0,u\}$.
    \end{remark}
    The H\"older continuity of polarizations of H\"older continuous functions defined on $\mathbb{R}^d$ is given in \cite[Corollary 3.1]{BrockSolynin2000}. The same result holds for the functions defined on a symmetric domain.
    \begin{proposition}\label{propo:pol_Lip}
	   Let $\Omega _0\subseteq \mathbb{R}^d$ be a domain and $H\in \mathscr{H}$ such that $\sigma _H(\Omega _0)=\Omega _0$. If $u\in \mathcal{C}^{0,\alpha } (\Omega _0)$ for some $\alpha \in (0,1]$, then $P_H u \in \mathcal{C}^{0,\alpha}(\Omega _0)$.
	\end{proposition}
	\begin{proof}
	   For $u\in \mathcal{C}^{0,\alpha }(\Omega _0)$, there exists $L>0$ such that $|u(x)-u(y)|\leq L |x-y|^\alpha \mbox{ for any } x,y \in \Omega _0.$ For simplicity of notation, we denote the reflection $\sigma_H(z)$ of $z\in \mathbb{R}^d$ with respect to $\partial H$ by $z^*$. Let $x,y\in \Omega _0$. Since $\sigma _H(\Omega _0)=\Omega _0,$ both $x^*, y^* \in \Omega _0$, and ${\rm supp}\left(P_H(u)\right)\subseteq \Omega _0$ (from Remark~\ref{rmk-1.7}). If both $x, y \in H$, then 
	   \begin{align*}
	       |P_H u(x)-P_H u(y)|&\leq \left|\max \big\{u(x), u(x^*)\big\}- \max \big\{u(y),u(y^*)\big\}\right|\\
	       &\leq \max \Big\{|u(x)-u(y)|, |u(x^*)-u(y^*)|\Big\}\leq L |x-y|^\alpha .
	   \end{align*}
	   Similarly, if $x,y\in H^\mathsf{c}$ then $|P_H u(x)-P_H u(y)|\leq L |x-y|^\alpha .$ Now, if $x\in H$ and $y\in H^\mathsf{c}$ then $|x-y^*|=|x^*-y|\leq |x-y|$. Therefore
	   \begin{align*}
	 	   |P_H u(x)-P_H u(y)|&\leq \left|\max \big\{u(x), u(x^*)\big\} - \min \big\{u(y), u(y^*)\big\}\right|\\
	 	   &\leq \max \Big\{|u(x)-u(y)|, |u(x)-u(y^*)|, |u(y)-u(x^*)|, |u(x^*)-u(y^*)| \Big\}\\
	 	   &\leq L |x-y|^\alpha . \qedhere
	  \end{align*}	
    \end{proof}
    Now, we state the following non-expansive property of polarization, see \cite[Theorem~3.1]{BrockSolynin2000} and \cite[Theorem 3, Corollary 1]{Crowe86}.
    \begin{proposition}\label{propo:pol_nonexp}
	   Let $\Omega _0\subset \mathbb{R}^d$, and $j$ be any Young function. Then, for any $H \in \mathscr{H}$ and any non-negative measurable functions $u,v$ on $\Omega _0$,
	   \begin{align*}
		  \int_{P_H\Omega _0}j\left(|P_H u - P_H v|\right) {\,\rm d}x \leq \int _{\Omega _0} j(|u-v|) {\,\rm d}x .
	   \end{align*}
	   In particular, for $j(t)=t^p, \ 1\leq p <\infty $,
	   \begin{align*}
		  \left\|P_H u -P_H v\right\|_{p,P_H\Omega _0} \leq \left\| u-v \right\|_{p,\Omega _0} \mbox{ for any non-negative } u, v\in L^p(\Omega _0).
	   \end{align*}
    \end{proposition} 
    We state the following invariance property of polarizations, see~\cite[Proposition 2.3.]{Schaftingen05} and \cite[Lemma~3.1]{Weth2010}.
    \begin{proposition}\label{propo:pol_invariant}
       Let $\Omega _0\subseteq \mathbb{R}^d$ be an open set and $H\in \mathscr{H}$ such that $\sigma _H(\Omega _0)=\Omega _0$. If $u\in W^{1,p}(\Omega _0)$ then $P_H(u) \in W^{1,p}(\Omega _0)$, and 
       \begin{align}\label{eqn:pol_norm}
          \left\|u\right\|_{p} = \left\|P_H u\right\|_{p} \mbox{ and } \left\|\nabla u\right\|_{p}= \left\|\nabla P_H u\right\|_{p}.
       \end{align}
    \end{proposition}
    \begin{proof}
	   Let $u\in W^{1,p}(\Omega _0)$. Since $\Omega _0$ is symmetric with respect to $\partial H$, we have $v\mathrel{\mathop:}=u\circ \sigma _H \in W^{1,p}(\Omega _0)$. Moreover, using  the standard arguments we can easily show that, $ |u-v|,$ $f\mathrel{\mathop:}=|u-v|\mathbb{1}_{\Omega _0 \cap H} ,$ and $g\mathrel{\mathop:}=-|u-v|\mathbb{1}_{\Omega _0 \cap H^\mathsf{c}}$ are in $W^{1,p}(\Omega _0) $. Thus $P_H(u)=\frac{1}{2}\left(u+v+f+g\right)$ is also in $W^{1,p}(\Omega _0)$. To prove that the norms are preserved, first observe that 
	   \begin{align*}
		P_H u =\left\{
		\begin{aligned}
			u \quad \mathrm{a.e.,}& \quad \mbox{ in } \left[\left(\Omega _0 \cap H\right)\cap \{u>v\}\right]\cup \left[\left(\Omega _0 \cap H^\mathsf{c} \right) \cap \{u<v\}\right],\\
			v \quad \mathrm{a.e.,}& \quad \mbox{ in } \left[\left(\Omega _0 \cap H^\mathsf{c} \right)\cap \{u>v\}\right]\cup \left[\left(\Omega _0 \cap H\right)\cap \{u<v\}\right];
		\end{aligned}
		\right.
	\end{align*}
	\begin{align*}
		\nabla P_H u =\left\{
		\begin{aligned}
			\nabla u \quad \mathrm{a.e.,}& \quad \mbox{in } \left[\left(\Omega _0 \cap H\right)\cap \{u>v\}\right]\cup \left[\left(\Omega _0 \cap {H}^\mathsf{c} \right)\cap \{u<v\}\right],\\
			\nabla v \quad \mathrm{a.e.,}& \quad \mbox{in } \left[\left(\Omega _0 \cap {H}^\mathsf{c} \right)\cap \{u>v\}\right]\cup \left[\left(\Omega _0 \cap H\right)\cap \{u<v\}\right].
		\end{aligned}\right.
	\end{align*}
	Now, by integrating $|P_H(u)|^p$ and $|\nabla P_H(u)|^p$ over $\Omega _0$, and using $\sigma _H\left(\left(\Omega _0 \cap H\right)\cap \{u>v\}\right) = \big(\Omega _0 \cap \overline{H}^\mathsf{c}\big)\cap \{u<v\}$ we get~\eqref{eqn:pol_norm}.
	\end{proof}
	Recall that, for a domain $\Omega \subseteq \mathbb{R}^d$ and $\Gamma _{\scriptscriptstyle D}\subseteq \partial \Omega$, the Sobolev space $W^{1,p}_{\Gamma _{\scriptscriptstyle D}}(\Omega )$ is defined by
	\begin{align*}
	W^{1,p}_{\scriptscriptstyle \Gamma _{\scriptscriptstyle D}}(\Omega )=\mbox{the closure of } \mathcal{C}^{0,1}_{\scriptscriptstyle \Gamma _{\scriptscriptstyle D}} (\Omega ) \mbox{ in } W^{1,p}(\Omega ),
	\end{align*}
	where $\mathcal{C}^{0,1}_{\scriptscriptstyle \Gamma _{\scriptscriptstyle D}} (\Omega ) = \bigl\{\varphi \in \mathcal{C}^{0,1}(\Omega ): {\rm supp} \left(\phi \right)\cap \Gamma _{\scriptscriptstyle D} = \emptyset \bigr\}$. We give the analogous result of Proposition~\ref{propo:pol_invariant} for the functions in $W^{1,p}_{\scriptscriptstyle \Gamma _{\scriptscriptstyle D}} (\Omega )$ in the following proposition.  
	\begin{proposition}\label{propo:pol_Lip1}
	Let $\Omega =\Omega_{\rm out}\setminus \overline{\Omega_{\rm in}} \subset \mathbb{R}^d$ be as given in~\ref{hypothesis}, $\Gamma _{\scriptscriptstyle D}\subseteq \partial \Omega $ and $H\in \mathscr{H}_{\rm ad}$. Let $\varphi  \in \mathcal{C}^{0,1}_{\Gamma _{\scriptscriptstyle D}}(\Omega )$ be any non-negative function. 
	\begin{enumerate}[label={\rm (\roman*)}]
		\item If $\Gamma _{\scriptscriptstyle D}=\partial \Omega_{\rm out}$ and $\sigma _H(\Omega _{\rm in}) = \Omega _{\rm in}$, then $P_H(\varphi )\in \mathcal{C}^{0,1}_{\scriptscriptstyle \partial P_H(\Omega_{\rm out})} (P_H(\Omega ))$. 
		\item If $\Gamma _{\scriptscriptstyle D} = \partial \Omega_{\rm in}$ and $\sigma _H (\Omega_{\rm out}) = \Omega_{\rm out}$, then $P_H(\varphi )\in \mathcal{C}^{0,1}_{\scriptscriptstyle \partial P^H(\Omega_{\rm in})} (P_H(\Omega ))$.
	\end{enumerate}
	In both of the cases~\eqref{eqn:pol_norm} holds.
	\end{proposition}
	\begin{proof}
	(i) Let $\Omega _0=\mathbb{R}^d \setminus \overline{\Omega _{\rm in}}$. Then $\Omega \subset \Omega _0$, and $\sigma _H(\Omega _0)=\Omega _0$. Let $\varphi \in \mathcal{C}^{0,1}_{\scriptscriptstyle \Gamma _{\scriptscriptstyle D}} (\Omega )$ be a non-negative function, and let $\widetilde{\varphi }$ be its zero extension to $\Omega _0$. Then $\widetilde{\varphi }\in \mathcal{C}^{0,1}(\Omega _0)$, and hence by  Proposition~\ref{propo:pol_Lip}, $P_H (\widetilde{\varphi })\in \mathcal{C}^{0,1}(\Omega _0)$. Therefore,  $P_H(\varphi ) = P_H(\widetilde{\varphi }) \mathbb{1}_{P_H(\Omega )} \in \mathcal{C}^{0,1}(P_H \Omega )$. Next, we show that $P_H(\varphi )=0$ on $\partial P_H(\Omega _{\rm out})$. Let $M=\mathrm{supp}\left(\varphi \right) \subsetneq \Omega _{\rm out}.$ Since ${\rm supp} \left(P_H(\varphi ) \right) \subseteq P_H(M)$ is closed, $P_H(\Omega _{\rm out})$ is open and $P_H(M)\subset P_H(\Omega _{\rm out})$, we obtain ${\rm supp} \left(P_H(\varphi )\right)\cap \partial P_H(\Omega _{\rm out})= \emptyset $ as required.
	
	\noindent (ii) In this case, let $\Omega_0=\Omega_{\rm out}.$ For a non-negative function $\varphi \in W^{1,p}_{\scriptscriptstyle \Gamma _{\scriptscriptstyle D}}(\Omega )$, as before we get $P_H(\varphi)= P_H(\widetilde{\varphi }) \mathbb{1}_{P_H(\Omega )} \in \mathcal{C}^{0,1}(P_H(\Omega ))$. Let $M={\rm supp}\left(\varphi\right).$ Then $M\cap \overline{\Omega_{\rm in}} = \emptyset $ and $M\subset \overline{\Omega_{\rm in}}^\mathsf{c}.$ Now, using Proposition~\ref{Propo-Obs} we obtain
	\begin{align*}
	    P_H(M)\subset P_H(\overline{\Omega_{\rm in}}^\mathsf{c})\subseteq P_H(\Omega_{\rm in}^\mathsf{c})=\left(P^H(\Omega _{\rm in})\right)^\mathsf{c}.
	\end{align*}
	Since ${\rm supp}\left(P_H(\varphi ) \right) \subseteq P_H(M)$ is closed, and $P^H(\Omega _{\rm in})$ is open, we get ${\rm supp} \left(P_H(\varphi )\right)\cap \partial P^H(\Omega _{\rm in})= \emptyset .$ Therefore, $P_H(\varphi )=0$ on $\partial P^H(\Omega _{\rm in})$.
	\end{proof}
	Using the standard approximation techniques and Proposition~\ref{propo:pol_nonexp} (the non-expansivity of polarizations), we can prove the following analogous result of Proposition~\ref{propo:pol_invariant}, for the functions in $W^{1,p}_{\scriptscriptstyle \Gamma _{\scriptscriptstyle D}}(\Omega )$.
	\begin{proposition}\label{propo-pol-invariant}
	Let $\Omega ,$ $\Gamma _{\scriptscriptstyle D}\subseteq \partial \Omega $ and $H$ be as given in Proposition~\ref{propo:pol_Lip1}. Let $u\in W^{1,p}_{\scriptscriptstyle \Gamma _{\scriptscriptstyle D}}(\Omega )$ be any non-negative function.
	\begin{enumerate}[label={\rm (\roman*)}]
		\item If $\Gamma _{\scriptscriptstyle D}=\partial \Omega_{\rm out}$ and $\sigma _H(\Omega _{\rm in}) = \Omega _{\rm in}$, then $P_H(u)\in W^{1,p}_{\scriptscriptstyle \partial P_H (\Omega_{\rm out})}(P_H(\Omega ))$. 
		\item If $\Gamma _{\scriptscriptstyle D}=\partial \Omega_{\rm in}$ and $\sigma _H (\Omega_{\rm out}) = \Omega_{\rm out}$, then $P_H(u)\in W^{1,p}_{\scriptscriptstyle \partial P^H (\Omega_{\rm in})}(P_H(\Omega ))$. 
	\end{enumerate}
	In both of the cases~\eqref{eqn:pol_norm} holds.
	\end{proposition} 
	\subsection{Regularity results and Strong comparison principles}
	Next, we recall a few regularity results for the eigenfunctions. Using Moser type iteration arguments~\cite[Proposition~1.2]{Barles88} and the arguments from~\cite[Remark~2.8]{Bensoussan02}, we get that the eigenfunctions are in $L^q$ for any $q\in [1,\infty ]$. Now, the local $\mathcal{C}^{1,\alpha }$-regularity results of~\cite[Theorem~1 and 2]{DiBenedetto83} give the following boundary regularity of the eigenfunctions.
	\begin{proposition}
	Let $\Omega \subseteq \mathbb{R}^d$ be a Lipschitz domain and $1<p<\infty .$ Let $u\in W^{1,p}_{\rm loc}(\Omega )\cap L^\infty _{\rm loc}(\Omega )$ be a weak solution of $-\Delta _p u = \lambda |u|^{p-2}u$ for some $\lambda \in \mathbb{R}$. Then there exists $\alpha \in (0,1)$ such that $u\in \mathcal{C}^{1,\alpha }_{\rm loc}(\Omega ) \cap \mathcal{C}^{0,\alpha } (\overline{\Omega })$.
	\end{proposition}
	The following strong comparison principle for the distributional solutions of the $p$-Laplace operator is given in~\cite[Theorem~1.4]{Sciunzi2014}.
	\begin{proposition}\label{SCP-p}
	Let $\Omega \subset \mathbb{R}^d$ be a bounded smooth domain and  $\frac{2d+2}{d+2}<p<\infty $. Let $u,v\in C^1(\overline{\Omega })$ be positive distributional solutions of $-\Delta _p u-g(u)=0$ in $\Omega $, for a non-negative Lipschitz function $g$ on $[0,\infty )$ with $g(s)>0$ for $s>0$. If $u\leq v$ in $\Omega $, then either $u<v$ in $\Omega $ or $u \equiv v$ in $\Omega $.
	\end{proposition}
    \section{Strict {F}aber-{K}rahn type inequality under polarization}\label{FK type}
    In this section, we give A proof for Theorem~\ref{thm~1.2}. Recall the following two subsets of $P_H(\Omega ) \cap H$: 
    \begin{align*}
	A_H =\Omega ^\mathsf{c} \cap \sigma _H(\Omega ) \cap H \mbox{ and } B_H = \Omega \cap \sigma _H (\Omega ^\mathsf{c}) \cap H.
	\end{align*} 
	We need the following lemma.
	\begin{lemma}\label{lemma~3.1}
	Let $\Omega=\Omega _{\rm out}\setminus \overline{\Omega _{\rm in}}  \subset \mathbb{R}^d$ be a domain as given in~\ref{hypothesis}, and $H\in \mathscr{H}_{\rm ad}$. Then $\overline{\Omega \cap H} \cap \overline{A_H}\subseteq \partial \Omega $. Furthermore, 
	\begin{enumerate}[label={\rm (\roman*)}]
		\item if $\sigma_H(\Omega _{\rm in})=\Omega _{\rm in}$ then $\overline{\Omega \cap H} \cap \overline{A_H}\subseteq \partial \Omega _{\rm out}$;
		\item if $\sigma_H(\Omega _{\rm out})=\Omega _{\rm out}$ then $\overline{\Omega \cap H} \cap \overline{A_H}\subseteq \partial \Omega _{\rm in}$.
	\end{enumerate}
	\end{lemma}
	\begin{proof}
	If $A_H =\emptyset $, then trivially $\emptyset = \overline{\Omega \cap H} \cap \overline{A_H} \subset \partial \Omega $. Let $A_H \neq \emptyset ,$ then from Proposition~\ref{propo:pol-equiv1}, we obtain $P_H(\Omega ) \neq \Omega $, and $A_H$ has non-empty interior. Since $\Omega \cap H$ and $A_H$ are disjoint and $P_H(\Omega )\cap H = \left(\Omega \cap H\right)\cup A_H$, using the connectedness of $P_H(\Omega )\cap H$ we conclude that $\overline{\Omega \cap H} \cap \overline{A_H} = \partial (\Omega \cap H) \cap \partial A_H \neq \emptyset $. Clearly $\Omega \cap \partial (\Omega \cap H) \cap \partial A_H = \emptyset $ and hence $\partial (\Omega \cap H) \cap \partial A_H \subseteq \partial \Omega .$
	
	\noindent
	(i) If $\sigma _H (\Omega _{\rm in}) = \Omega _{\rm in}$, then we can write 
	\begin{align*}
		A_H= \Omega ^\mathsf{c} \cap \sigma _H(\Omega )\cap H &= \left( \Omega _{\rm out}^\mathsf{c} \cup \Omega _{\rm in} \right)\cap \sigma _H \left(\Omega _{\rm out} \right) \cap \Omega _{\rm in}^\mathsf{c} \cap H \\
		& = \Omega _{\rm out}^\mathsf{c} \cap \sigma _H \left(\Omega _{\rm out} \right) \cap H.
	\end{align*}
	Since $\Omega _{\rm in} \subset \subset \Omega _{\rm out}$, we have $\partial \Omega = \partial \Omega_{\rm out} \sqcup \partial \Omega_{\rm in}$ and $\partial \Omega _{\rm in}\cap \partial A_H =\emptyset $. Therefore, $\overline{\Omega \cap H} \cap \overline{A_H} \subset \partial \Omega _{\rm out}.$
	
	\noindent
	(ii) Similarly, for $\sigma _H (\Omega _{\rm out})=\Omega _{\rm out}$ we have $A_H=\Omega _{\rm in} \cap \sigma _H (\Omega _{\rm in})\cap H$. Therefore, we obtain $\overline{\Omega \cap H} \cap \overline{A_H} \subset \partial \Omega _{\rm in}.$
	\end{proof}
	For any non-negative function $u\in \mathcal{C}(\overline{\Omega }),$ let $\widetilde{u}$ be its zero extension to $\mathbb{R}^d$ and let 
	\begin{align*}
	M_u=\Big\{x\in P_H(\Omega )\cap H : P_H(\widetilde{u})(x)>\widetilde{u}(x) \Big\}.
	\end{align*}
	Next, we prove a lemma that plays a significant role in our results.
	\begin{lemma}\label{lem:3.3}
	Let $\Omega =\Omega _{\rm out}\setminus \overline{\Omega _{\rm in}} \subset \mathbb{R}^d$ be a domain as given in~\ref{hypothesis}, $\Gamma _{\scriptscriptstyle D}\subseteq \partial \Omega $, and $H\in \mathscr{H}_{\rm ad}$. Let $u\in \mathcal{C}(\overline{\Omega})$ be a non-negative function with $u=0$ on $\Gamma _{\scriptscriptstyle D}$. If $\Gamma _{\scriptscriptstyle D}$ satisfies one of the following assumptions:
	\begin{align*}
		(a)\;\Gamma _{\scriptscriptstyle D}=\partial \Omega,\quad (b)\; 
		\Gamma _{\scriptscriptstyle D}=\partial \Omega _{\rm out}\mbox{ and } \sigma_H(\Omega _{\rm in})=\Omega _{\rm in}, \quad (c)\; \Gamma _{\scriptscriptstyle D}=\partial \Omega _{\rm in}\mbox{ and } \sigma_H(\Omega _{\rm out})=\Omega _{\rm out},
	\end{align*}
	then $\widetilde{u}$ is continuous on $P_H(\Omega ) \cap H$. Moreover, if $\Omega \neq P_H(\Omega ) \neq \sigma_H(\Omega)$ then there exists a ball $B_0 \subset \Omega \cap H$ such that
	\begin{align*}
		P_H(u)>u \mbox{ in } B_0 \cap M_u \mbox{ and } P_H(u)\equiv u \mbox{ in } B_0 \cap M_u^\mathsf{c}.
	\end{align*}
	\end{lemma}
	\begin{proof}
	If $P_H(\Omega )=\Omega $ then $\widetilde{u}=u$ in $P_H(\Omega ) \cap H$, and hence it is continuous. If $P_H(\Omega ) \neq \Omega $, then from Proposition~\ref{propo:pol-equiv1} we have $A_H \neq \emptyset $ and $\Omega \cap H \subsetneq P_H(\Omega )\cap H$. Clearly $\widetilde{u}=u$ on $\Omega\cap H$ and $\widetilde{u}=0$ on $A_H$, and hence $\widetilde{u}$ is continuous on both $\Omega \cap H$ and $A_H$. If $\Gamma _{\scriptscriptstyle D}$ satisfies one of the assumptions $(a)$-$(c)$, then by Lemma~\ref{lemma~3.1} we get $\emptyset \neq \overline{\Omega \cap H} \cap \overline{A_H}\subseteq \Gamma _{\scriptscriptstyle D}$. Therefore, $\widetilde{u}=u=0$ on $\overline{\Omega \cap H} \cap \overline{A_H}$ and hence $\widetilde{u}$ is continuous on $\left(\Omega \cap H \right) \cup A_H = P_H (\Omega)\cap H$. 
	
	\noindent
	Now assume that $\Omega \neq P_H(\Omega )\neq \sigma_H(\Omega )$. Then from Proposition~\ref{propo:pol-equiv1}, both $A_H$ and $B_H$ have non-empty interiors. By the definition of $P_H (\widetilde{u})$, we get $P_H(\widetilde{u})\geq \widetilde{u} \text{ in  } P_H(\Omega)\cap H$, and 
	\begin{align}\label{inequality}
		\begin{aligned}
			\text{ in  } A_H &: \widetilde{u}=0, \widetilde{u}\circ \sigma _H = u \circ \sigma _H > 0 \mbox{ and hence } P_H(\widetilde{u}) = u \circ \sigma _H >\widetilde{u};\\
			\mbox{ in } B_H &:\widetilde{u}=u>0,\widetilde{u}\circ \sigma _H =0 \mbox{ and hence } P_H(\widetilde{u})=\widetilde{u}.
		\end{aligned}
	\end{align}
	Let $N=\big\{x\in P_H(\Omega )\cap H: P_H(\widetilde{u})(x)=\widetilde{u}(x)  \big\}$.  Since $P_H(\widetilde{u})$ is  also continuous  on $P_H(\Omega ) \cap H$ (Proposition~\ref{propo:pol_Lip}), from \eqref{inequality} we get  $N \subsetneq \Omega \cap H$ is a non-empty closed set and $M_u = (P_H(\Omega ) \cap H) \setminus N$ is a non-empty open set in $P_H(\Omega )\cap H$. Now, by the connectedness of $P_H(\Omega )\cap H$ we must have $\partial M_u \cap N \neq \emptyset$. For $x_0\in \partial M_u \cap N$, let  $B _0 = B_r(x_0) \subset \Omega \cap H$. Then  $B_0$ has all the the desired properties.    	
	\end{proof}
	Now, we prove Theorem~\ref{thm~1.2}.
	\begin{proof}[Proof of Theorem~\ref{thm~1.2}]
	Let $1<p<\infty $, $\Omega _{\rm out}\setminus \overline{\Omega _{\rm in}} \subset \mathbb{R}^d$ be as given in~\ref{hypothesis} and $H\in \mathscr{H}_{\rm ad}$. Denote $\Omega =\Omega _{\rm out}\setminus \overline{\Omega _{\rm in}}$.
	
	\noindent
	(i) Assume that $\sigma_H(\Omega _{\rm in})=\Omega _{\rm in}$. Let $0<u \in \mathcal{C}^{0,\alpha }_{\scriptscriptstyle \partial P_H(\Omega _{\rm out})} (\overline{\Omega })$ be an eigenfunction corresponding to $\nu _1(\Omega )$. Define $v=P_H(u)$ in $P_H(\Omega )$ then, from~Proposition~\ref{propo:pol_Lip1}, we get $v\in \mathcal{C}^{0,\alpha } _{\scriptscriptstyle \partial P_H(\Omega _{\rm out})}(P_H (\Omega ))$, and
	\begin{align*}
		\left\| u \right\| _{p,\Omega } = \left\| v \right\| _{p,P_H(\Omega )} \mbox{ and } \left\| \nabla u \right\| _{p,\Omega } = \left\| \nabla v \right\| _{p,P_H(\Omega )}.
	\end{align*}
	From the variational characterization of $\nu _1 (P_H(\Omega ))$, we obtain:
	\begin{align}\label{eqn:pol_eigen1}
		\nu _1(P_H(\Omega ))\leq \nu _1(\Omega).
	\end{align}
	\noindent
	(ii) Assume that $\sigma_H(\Omega _{\rm out})=\Omega _{\rm out}$. Let $0<u \in \mathcal{C}^{0,\alpha }_{\scriptscriptstyle \partial P^H(\Omega _{\rm in})} (\overline{\Omega })$ be an eigenfunction corresponding to $\tau _1(\Omega )$. Define $v=P_H(u)$ in $P_H(\Omega )$ then, from~Proposition~\ref{propo:pol_Lip1}, we obtain $v\in \mathcal{C}^{0,\alpha } _{\scriptscriptstyle \partial P^H(\Omega_{\rm in})}(P_H (\Omega ))$, and  
	\begin{align*}
		\left\| u \right\| _{p,\Omega } = \left\| v \right\| _{p,P_H(\Omega )} \mbox{ and } \left\| \nabla u \right\| _{p,\Omega } = \left\| \nabla v \right\| _{p,P_H(\Omega )}.
	\end{align*}
	From the variational characterization of $\tau _1 (P_H(\Omega ))$, we get
	\begin{align}\label{eqn:pol_eigen11}
		\tau _1(P_H(\Omega ))\leq \tau _1(\Omega).
	\end{align}
	
	\noindent 
	(iii) Let $\frac{2d+2}{d+2}<p<\infty.$ Assume that, the equality holds in~\eqref{eqn:pol_eigen1}. Let $0\leq u \in \mathcal{C}^{0,\alpha } _{\scriptscriptstyle \partial P^H(\Omega _{\rm out})} (\overline{\Omega })$ be an eigenfunction corresponding to the eigenvalue $\nu _1(\Omega )$. On the contrary, assume that $\Omega \neq P_H(\Omega )\neq \sigma _{\scriptscriptstyle H}(\Omega )$. Then by Lemma~\ref{lem:3.3}, there exists a ball $B_0 \subset \Omega \cap H$ such that 
	\begin{align}\label{ineq2}
		v>u \mbox{ in } B_0 \cap M_u \mbox{ and } v\equiv u \mbox{ in } B_0 \cap M_u^\mathsf{c},
	\end{align}
	where $M_u=\big\{x\in P_H(\Omega )\cap H : v(x)> u(x)\big\}$ is a non-empty open set. Then, both $u,v\in \mathcal{C}^1(\overline{B_0})$ are positive distributional solutions for the following problem in $B_0$:
	\begin{align*}
		-\Delta _p u - \lambda |u|^{p-2}u=0\mbox{ in } B_0.
	\end{align*}
	Now, for $\frac{2d+2}{d+2}<p<\infty $, the strong comparison principle (Proposition~\ref{SCP-p}) implies that 
	\begin{align*}
		\mbox{ either } u<v \mbox{ or } u\equiv v \mbox{ in } B_0.
	\end{align*}
	This is a contradiction to \eqref{ineq2}, and hence we must have
	$P_H(\Omega )=\Omega \mbox{ or } P_H(\Omega )=\sigma _{\scriptscriptstyle H}(\Omega ).$
	If the equality holds in \eqref{eqn:pol_eigen11}, the proof will follow using a similar set of arguments as given above.
	\end{proof}
    \section{Strict monotonicity of first eigenvalues via polarization}\label{MonoEigenvalue}
    In this section, we prove Theorem~\ref{thm-mono} and Theorem~\ref{thm-circ}. The main idea is to express the translations and the rotations of the obstacle $\mathscr{O}$ in terms of polarizations and apply the Faber-Krahn type inequality to get the desired monotonicity. 
    \subsection{Monotonicity along a straight line} 
    Now, we give a proof for Theorem~\ref{thm-mono}. First, we recall that:
    \begin{align*}
	\mbox{for given } h\in\mathbb{S}^{d-1}, \quad H_s \mathrel{\mathop:}= \big\{x\in \mathbb{R}^d : x\cdot h<s\}, \quad \Sigma _s \mathrel{\mathop:}= \left\{ x\in \Omega : x\cdot h\geq s \right\}, \mbox{for $s \in \mathbb{R}$,} 
	\end{align*}
	$P_{H_0}(\Omega )=\Omega $, the obstacle $\mathscr{O}$ is Steiner symmetric with respect to $\partial H_0$, and the translations of $\mathscr{O}$ in the $h$-direction are given by $\mathscr{O}_s = s h +\mathscr{O}$ for $s \in \mathbb{R},$ and \begin{align*}
	    \mathrm{L}_\mathscr{O}\mathrel{\mathop:}= \Big\{s \in \mathbb{R} : P_{\scriptscriptstyle H_s} (\Omega ) = \Omega \mbox{ and } \mathscr{O}_s \subset \Omega \Big\}.
	\end{align*}
	We observe the following facts:
	\begin{align}
	    \mbox{ for } x\in \mathbb{R}^d, \quad \sigma _{H_0}(x)=x-2(x\cdot h)h, \mbox{ and } \sigma _{H_s}(x)=2sh+\sigma _{H_0}(x) \mbox{ for } s\in \mathbb{R}.
	\end{align}
	\begin{lemma}\label{lemma-4.1}
	If $\mathscr{O}\subseteq \mathbb{R}^d$ satisfies $\sigma _{\scriptscriptstyle H_0}(\mathscr{O})=\mathscr{O}$, then  for any $s, t\in \mathbb{R}$, $th+ \mathscr{O}_s=\mathscr{O}_{s+t}$ and $\sigma _{H_t}(\mathscr{O}_s)=\mathscr{O}_{2t-s}$.
	\end{lemma}
	\begin{proof}
	It is easy to verify that $t h+ \mathscr{O}_s=(s+t)h+\mathscr{O}=\mathscr{O}_{s+t}$ and $\sigma _{H_s}(\mathscr{O}_s)=\mathscr{O}_s$. Since $\sigma _{H_t}(x)=2t h+ \sigma _{H_0}(x)=2(t-s)h+2sh+ \sigma _{H_0}(x) = 2(t-s)h + \sigma _{H_s}(x)$ any $x\in \mathbb{R}^d$, we get $\sigma _{H_t}(\mathscr{O}_s)=2(t-s) h + \sigma_{H_s}(\mathscr{O}_s)=2(t-s) h + \mathscr{O}_s = \mathscr{O}_{2t-s}$.
	\end{proof}
	\begin{proof}[Proof of Theorem~\ref{thm-mono}]
	Let the set $\Sigma _{s_0} \bigcup \sigma _{H_{s_0}} \left(\Sigma _{s_0} \right)$ is convex in the $h$-direction for some $s_0 \in \mathrm{L}_\mathscr{O}$. Let $\mathrm{R}_\mathscr{O}:=\sup \mathrm{L}_\mathscr{O}$.
	
	\medskip
	\noindent
	\underline{\textbf{The interval $[s_0, \mathrm{R}_\mathscr{O}) \subseteq \mathrm{L}_\mathscr{O}$:}} Let $s\in (s_0, \mathrm{R}_\mathscr{O})$. Clearly, $\Sigma _s \subset \Sigma _{s_0}$ and the convexity of the set $\Sigma _{s_0} \bigcup \sigma _{\scriptscriptstyle H_{s_0}}(\Sigma_{s_0})$ in the $h$-direction implies that $\sigma _{\scriptscriptstyle H_s} (\Sigma _s)\subset \Sigma _{s_0}\bigcup \sigma _{\scriptscriptstyle H_{s_0}} (\Sigma _{s_0}) \subset \Omega $ (see Proposition~\ref{propo:char}-(i)). Therefore, by  Proposition~\ref{propo:pol-equiv}, we get $P_{\scriptscriptstyle H_s}(\Omega ) = \Omega $. Next, we show $\mathscr{O}_s\subset \Omega $.  By the definition of $\mathrm{R}_\mathscr{O}$, there exists $s_1 \in \mathrm{L}_\mathscr{O}$ such that $s< s_1 <\mathrm{R}_\mathscr{O}$.  Observe that $\mathscr{O}_{s_0}, \mathscr{O}_{s_1}\subset \Sigma _{s_0} \bigcup \sigma _{\scriptscriptstyle H_{s_0}} (\Sigma _{s_0})$, and  $s=t s_0+(1-t) s_1$, for some $t\in (0,1).$ Thus $\mathscr{O}_s = t \mathscr{O}_{s_0}+(1-t)\mathscr{O}_{s_1}$, and hence the convexity of $\Sigma _{s_0} \bigcup \sigma _{\scriptscriptstyle H_{s_0}} (\Sigma _{s_0})$ in the $h$-direction implies that $\mathscr{O}_s \subset \Omega $. Therefore $s\in \mathrm{L}_\mathscr{O}$.
	
	\medskip
	\noindent
	\underline{\textbf{Monotonicity of $\lambda _1(\cdot)$ on $[s_0, \mathrm{R}_\mathscr{O})$:}}
	Let $s<t$ in $[s_0, \mathrm{R}_\mathscr{O})$. Then $\overline{s}= \frac{s+t}{2}\in\mathrm{L}_\mathscr{O}$ and hence $P_{\scriptscriptstyle H_{\overline{s}}}(\Omega )=\Omega $. Since $\mathscr{O}_{s}$ is Steiner symmetric with respect to $\partial H_{s}$ and $\overline{s}>s$, from Proposition~\ref{propo:char}-(i), we get $P^{\scriptscriptstyle H_{\overline{s}}} (\mathscr{O}_{s}) = \sigma _{\scriptscriptstyle H_{\overline{s}}} (\mathscr{O}_{s})$. From Lemma~\ref{lemma-4.1} (since $\sigma _{H_0}(\mathscr{O})=\mathscr{O}$) we also have $\sigma _{\scriptscriptstyle H_{\overline{s}}} (\mathscr{O}_{s}) = \mathscr{O}_{2\overline{s}-s} = \mathscr{O}_{t}$. Therefore, from Proposition~\ref{propo:pol-compliment1} we obtain  
	\begin{align*}
		P_{\scriptscriptstyle H_{\overline{s}}}(\Omega \setminus \mathscr{O}_{s}) = P_{\scriptscriptstyle H_{\overline{s}}}(\Omega ) \setminus P^{\scriptscriptstyle H_{\overline{s}}}(\mathscr{O}_{s}) = \Omega \setminus \mathscr{O}_{t}.
	\end{align*}
	For $\frac{2d+2}{d+2}<p<\infty $, the Faber-Krahn type inequality (Theorem~\ref{thm~1.2}) implies that $\lambda_{1}(t) < \lambda_{1}(s)$. Therefore, the first Dirichlet eigenvalue $\lambda _1(\cdot )$ is strictly decreasing on $[s_0, \mathrm{R}_\mathscr{O})$. 
	\end{proof}
	\begin{remark}
	If we drop the convexity assumption from Theorem~\ref{thm-mono}, then $L_0$ might not be an interval. However, for any $s,t \in L_0$ with $\frac{s+t}{2} \in L_0$ and $s<t$, the above proof still yields $\lambda _1(t)<\lambda _1(s).$
	\end{remark}
	\subsection{Monotonicity with respect to the rotations about a point}
	Now, we prove Theorem~\ref{thm-circ}. First recall that, for $\xi \in \mathbb{S}^{d-1}\setminus \{\eta \}$ the rotations of the obstacle $\mathscr{O}$ with the  plane of rotation is $X_{\xi }=\mbox{span}\left\{\eta ,\xi \right\}$ about the point $a\in \mathbb{R}^{d}$ are given by: for $s \in [-1,1],$
	\begin{align*}
	\mathscr{O}_{s,\xi }:= a+R_{s,\xi }(-a+\mathscr{O}),
	\end{align*}
	where $R_{s,\xi }$ is the simple rotation in $\mathbb{R}^d$ with  $X_\xi $ as the plane of rotation and $\theta _s =\arccos (s)\in [0,\pi ]$ as the angle of rotation from the ray $\mathbb{R}^+\eta $ in the counter-clockwise direction.	We prove the following lemmas. 
	\begin{lemma}\label{lemma:fssObs}
	For any distinct $\xi _1, \xi _2 \in \mathbb{S}^{d-1}\setminus \{\eta \},$
	there exists a simple rotation $R$  such that
	\begin{align*}
	   R(-a+\Omega\setminus \mathscr{O}_{s,\xi _1})&=-a+\Omega \setminus \mathscr{O}_{s,\xi_2}.
	\end{align*}
	\end{lemma}
	\begin{proof}
		Let $\xi _1, \xi _2 \in \mathbb{S}^{d-1}\setminus \{\eta \}$, define
		\begin{align}
	    \widetilde{\xi _i}=\frac{\xi _i -(\xi _i \cdot \eta ) \eta}{\left\|\xi _i-(\xi _i \cdot \eta )\eta \right\|} \mbox{ for } i=1,2.
	    \end{align}
	    Observe that, the rotation of $\eta $ under $R_{s, \xi _i}$ is given by $R_{s, \xi _i} (\eta) = s\eta+\sqrt{1-s^2}\,\widetilde{\xi _i}$ for $i=1,2$. Consider the plane $X= \mbox{span}\left\{\widetilde{\xi_1}, \widetilde{\xi_2}\right\}$, that is  orthogonal to $\eta $. Let $R$ be the simple rotation such that $R(\widetilde{\xi_1})=\widetilde{\xi_2}.$ Thus $R$ must fix $\eta$, and 
		\begin{align*}
		    R\circ R_{s,\xi _1}(\eta )= R(s\eta+\sqrt{1-s^2}\,\widetilde{\xi_1})=s\eta+\sqrt{1-s^2}\,\widetilde{\xi_2}=R_{s,\xi _2}(\eta ).
		\end{align*}
		Therefore, for $r>0$, $\rho >0$,
		\begin{align*}
		R\left(B_{\rho }(rR_{s,\xi _1}(\eta ))\cap \partial B_r(0)\right)=B_{\rho }(rR\circ R_{s,\xi _1} (\eta ))\cap \partial B_r(0)=B_{\rho }(rR_{s,\xi _2} (\eta ))\cap \partial B_r(0),
		\end{align*}
		and from~\eqref{eqn:fssSet} we obtain $R\left(R_{s,\xi _1}(-a+\mathscr{O})\right)=R_{s,\xi _2}(-a+\mathscr{O})$, and hence $R\left(-a+\mathscr{O}_{s,\xi _1}\right)=-a+\mathscr{O}_{s,\xi _2}.$ Since $R$ fixes $\eta$, and $\Omega $ is foliated Schwarz symmetric with respect to $a+\mathbb{R}^+\eta $, we get $R(-a+\Omega ) = -a+\Omega $. Thus  we obtain
		\begin{align*}
		R(-a+\Omega\setminus \mathscr{O}_{s,\xi _1})&= R(-a+\Omega ) \setminus R(-a+\mathscr{O}_{s,\xi_1})\\
		&=(-a+\Omega )\setminus (-a+\mathscr{O}_{s,\xi_2})=-a+\Omega \setminus \mathscr{O}_{s,\xi_2}.
		\end{align*}
	\end{proof}
	From Lemma~\ref{lemma:fssObs}, we only need to consider the rotations of the obstacle by $R_{s,\xi}$ about the point $a$ in a $X_\xi$-plane for a fixed $\xi  \in \mathbb{S}^{d-1}\setminus \{\eta \}.$ Thus for $s\in [-1,1],$ we set $\mathscr{O}_s=\mathscr{O}_{s,\xi}$. Recall that:
	\begin{align*}
	\mbox{for } a\in \mathbb{R}^d \mbox{ and }  \eta \in \mathbb{S}^{d-1}, \quad \mathscr{H}_{a,\eta }=\Big\{H\in \mathscr{H} : a\in \partial H \mbox{ and } a+\mathbb{R}^+ \eta \subset H \Big\}.
	\end{align*}
	\begin{lemma}\label{lemma-4.2}
	Let $a\in \mathbb{R}^d$, $\eta \in \mathbb{S}^{d-1}$, and $\mathscr{O}\subset\mathbb{R}^d$ is foliated Schwarz symmetric with respect to the ray $a+\mathbb{R}^+\eta $. Let $\xi \in \mathbb{S}^{d-1}\setminus \{\eta \}$ and the rotations of $\mathscr{O}$ be as given in~\eqref{obstacle}. Then, for any $s<t$ in $[-1,1]$ there exists $H\in \mathscr{H} _{a, \eta }$ such that 
	\begin{align*}
	    \mbox{\rm (a) } \sigma _H (\mathscr{O}_{t}) = \mathscr{O}_{s},\quad \mbox{\rm (b) } P^H(\mathscr{O}_{t}) = \mathscr{O}_{s} \quad \mbox{and } \mbox{\rm (c) } P_H(\mathscr{O}_s)=\mathscr{O}_t.
	\end{align*}
	\end{lemma}
	\begin{proof}
	Let $h=R_s(\eta )-R_t(\eta )$ and consider the polarizer $H:=\bigl\{z\in \mathbb{R}^d: (z-a)\cdot h <0 \bigr\}$. Observe that $a\in \partial H$, and for $r>0$,
	\begin{align*}
	    (a+r\eta -a)\cdot h&=r\eta \cdot [R_s(\eta )-R_t(\eta )] =r(s-t)<0,\\
	    (a+rR_{t}(\eta )-a)\cdot h &= r R_t(\eta )\cdot [R_s(\eta ) -R_t(\eta )] =r[R_s(\eta )\cdot R_t(\eta )-1]<0.
	\end{align*}
	Therefore $H\in \mathscr{H}_{a,\eta }\cap \mathscr{H}_{a,R_t(\eta )}$.
	
	\noindent
	(a) Notice that, $\|h\|=2[1-R_s(\eta )\cdot R_t(\eta )]$.  Now, for $x=a+rR_t(\eta )$, $r>0$ we get
	\begin{align*}
	    \sigma _H(x)&=x-\frac{2(x-a)\cdot h}{\|h\|^2}h=
	    a+rR_t(\eta )-\frac{2r[R_s(\eta )\cdot R_t(\eta )-1]}{2[1-R_s(\eta )\cdot R_t(\eta )]}(R_s(\eta )-R_t(\eta ))= a+rR_s(\eta).
	\end{align*}
	Therefore, $\sigma _H(a+\mathbb{R}^+ R_t(\eta ))=a+\mathbb{R}^+ R_s(\eta )$, and hence from~\eqref{eqn:fssSet} we obtain
	\begin{align*}
	        \sigma _H(\mathscr{O}_t)&=\bigcup_{r\in I_\mathscr{O}}  B_{\rho(r)}(\sigma _H(a+rR_t(\eta )))\cap \partial B_r(\sigma _H(a))
	        =\bigcup_{r\in I_\mathscr{O}}  B_{\rho(r)}(a+rR_s(\eta ))\cap \partial B_r(a)=\mathscr{O}_s.
	\end{align*}
	(b) Since $H\in \mathscr{H}_{a,R_t(\eta )}$ and $\mathscr{O}_t$ is foliated Schwarz symmetric with respect to $a+\mathbb{R}^+R_{t}(\eta )$, Proposition~\ref{propo:char}-\ref{char:fss} implies that $P^H(\mathscr{O}_{t})=\sigma _H (\mathscr{O}_{t}) = \mathscr{O}_{s}$. 
	
	\noindent
	(c) Since $P_H(\sigma _H (\mathscr{O}_{t}))=P_H(\mathscr{O}_{t})$ (Proposition~\ref{Propo-Obs}-(iv)), we get $P_H(\mathscr{O}_s)=P_H(\sigma _H (\mathscr{O}_{t}))=P_H(\mathscr{O}_{t})=\mathscr{O}_{t}$.
	\end{proof}
	\begin{proof}[Proof of Theorem~\ref{thm-circ}] 
	Given $\frac{2d+2}{d+2}<p<\infty $, $\Omega $ and $\mathscr{O}$ are foliated Schwarz symmetric with respect to $a+\mathbb{R}^+\eta $.
	
	\medskip
	\noindent
	\underline{\textbf{The set $\mathrm{C}_\mathscr{O}$ is interval:}} We show that, for  any $s\in C_\mathscr{O}$, the interval $[s,1]\subseteq C_\mathscr{O}$. Let $t\in (s,1]$. From Lemma~\ref{lemma-4.2}-(c) there exists $H\in \mathscr{H}_{a,\eta }$ such that $P_H(\mathscr{O}_{s})=\mathscr{O}_{t}$. Since $H\in \mathscr{H}_{a,\eta }$ and $\Omega $ is foliated Schwarz symmetric with respect to $a+\mathbb{R}^+ \eta $, from Proposition~\ref{propo:char}-\ref{char:fss}, we get $P_H(\Omega )=\Omega $. Now, $\mathscr{O}_s\subset \Omega$ implies that $\mathscr{O}_t= P_H(\mathscr{O}_s)\subset P_H(\Omega )=\Omega $. Therefore, $t\in C_\mathscr{O}$ and hence $[s,1]\subseteq C_\mathscr{O}$.
	
	\medskip
	\noindent 
	\underline{\textbf{Monotonicity of $\gamma _1(\cdot )$}:}
	Let $s<t$ in $\mathrm{C}_\mathscr{O}$. From Lemma~\ref{lemma-4.2}, there exists $H\in \mathscr{H}_{a, \eta }$ such that $P^H(\mathscr{O}_{t}) = \mathscr{O}_{s}$. Since $\Omega $ is foliated Schwarz symmetric with respect to $a+\mathbb{R}^+ \eta $, from Proposition~\ref{propo:char}-\ref{char:fss}, we have $P_H(\Omega )=\Omega $, and from Proposition~\ref{propo:pol-compliment1} we get
	\begin{align*}
		P_H(\Omega \setminus \mathscr{O}_{t})=P_H(\Omega ) \setminus P^H(\mathscr{O}_{t}) = \Omega \setminus \mathscr{O}_{s}.
	\end{align*}
	If $\Omega $ satisfies~\ref{choice-1} then $\Omega =\Omega _0 \setminus \overline{B}_{\rho }(a)$ and $\Gamma _N=\partial B_\rho (a)$. In this case, we set $\Omega _{\rm out}=\Omega _0 \setminus \mathscr{O}_{t}$ and $\Omega _{\rm in}=B_{\rho _0}(a)$ (in Theorem~\ref{thm~1.2}) so that $\Omega _{\rm out}\setminus \overline{\Omega _{\rm in}}=\Omega \setminus \mathscr{O}_t$, $\Gamma _N=\partial \Omega _{\rm in}$ and $\Gamma _D=\partial \Omega _{\rm out}$. Therefore, we have $\nu _1(\Omega _{\rm out}\setminus \overline{\Omega _{\rm in}})=\gamma _1(\Omega \setminus \mathscr{O}_t)$ and $\nu _1(P_H(\Omega _{\rm out}\setminus \overline{\Omega _{\rm in}}))=\gamma _1(\Omega \setminus \mathscr{O}_s)$. Since $\sigma _H(\Omega _{\rm in})=\Omega _{\rm in}$, from Theorem~\ref{thm~1.2}-(i) we get 
	\begin{align*}
	    \gamma _1(\Omega \setminus \mathscr{O}_s)\leq \gamma _1(\Omega \setminus \mathscr{O}_t).
	\end{align*}
	
	\noindent
	Similarly, if $\Omega $ satisfies~\ref{choice-2} then $\Omega = B_R(a)\setminus \overline{\Omega _1}$ and $\Gamma _N=\partial B_R(a)$. In this case, we set $\Omega _{\rm out} = B_R(a)$ and $\Omega _{\rm in} = \Omega _1 \cup \mathscr{O}_{t}$ (in Theorem~\ref{thm~1.2}) so that $\Omega _{\rm out}\setminus \overline{\Omega _{\rm in}}=\Omega \setminus \mathscr{O}_t$, $\Gamma _N=\partial \Omega _{\rm out}$ and $\Gamma _D=\partial \Omega _{\rm in}$. Therefore, we have $\tau _1(\Omega _{\rm out}\setminus \overline{\Omega _{\rm in}})=\gamma _1(\Omega \setminus \mathscr{O}_t)$ and $\tau _1(P_H(\Omega _{\rm out}\setminus \overline{\Omega _{\rm in}}))=\gamma _1(\Omega \setminus \mathscr{O}_s)$. Since $\sigma _H({\Omega _{\rm out}})=\Omega _{\rm out}$, from Theorem~\ref{thm~1.2}-(ii) we get
	\begin{align*}
	    \gamma _1(\Omega \setminus \mathscr{O}_s)\leq \gamma _1(\Omega \setminus \mathscr{O}_t).
	\end{align*}
	
	\noindent
	Since $\Omega $ is not radially symmetric with the center $a$, in both cases, we have $\Omega _{\rm out} \setminus \overline{\Omega_{\rm in}} \neq P_H \left(\Omega _{\rm out} \setminus \overline{\Omega_{\rm in}}\right)\neq \sigma _H \left(\Omega _{\rm out} \setminus \overline{\Omega_{\rm in}} \right)$. Thus, the strict Faber-Krahn type inequality (Theorem~\ref{thm~1.2}-(iii)) implies 
	\begin{align*}
		\gamma _{1}(s)<\gamma _{1}(t).
	\end{align*}
	Therefore, $\gamma _1 (\cdot )$ is strictly increasing on $\mathrm{C}_\mathscr{O}$.
	\end{proof}
	\section{Some remarks and examples}\label{Remarks}

	\begin{example}\label{ex1}
	Let $\Omega \subset \mathbb{R}^2$ is given by $\Omega = \bigl\{(x,y):x^2+y^2<R^2, x\leq 0\bigr\}\cup \big\{(x,y): |x|+|y|<2R, x\geq 0\bigr\}$ for $R>0$, and $\mathscr{O}$ is the rhombus given by $|x|+|y|\leq 2\ell $ with $\ell <R$ (see Figure~\ref{fig:ex1}).	Since $\mathscr{O}$ is Steiner symmetric with respect to the hyperplanes $S_1:=\{(x,y)\in \mathbb{R}^2:x=0\}$ and $S_2:=\{(x,y)\in \mathbb{R}^2: x+y=0\}$, we can consider the translations $\mathscr{O}$ along the $x$-axis, as well as along the straight line $y=x$.
	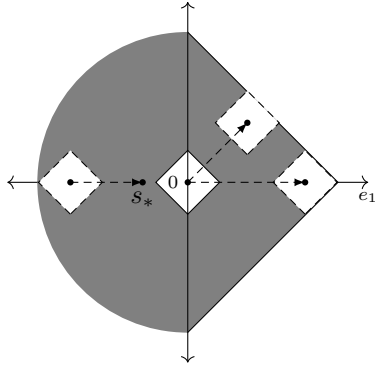
\begin{figure}[tbh!]
	    \centering
	    \begin{tikzpicture}[scale=0.4]
	        \draw [<->] (-6,0) -- (6,0);
			\draw (6,0) node[anchor=north] {$\scriptstyle e_1$};
	        \node [diamond, fill=gray, draw, minimum width = 4cm, minimum height=4cm] (d) at (0,0){};
	        \begin{scope}
	            \clip (0,0) ellipse (5cm and 5cm);
	            \clip [rotate=90] (-5,0) rectangle (5,5);
	            \fill [rotate=90,color=gray] (-5,0) rectangle (5,5);
	        \end{scope}
			\draw [rotate=45,fill=white] (-0.75,-0.75) rectangle (0.75,0.75);
			\draw [densely dashed,fill=white,xshift=3.9cm,rotate around ={45:(0,0)}] (-0.75,-0.75) rectangle (0.75,0.75);
			\draw [densely dashed,fill=white,xshift=-3.9cm,rotate around ={45:(0,0)}] (-0.75,-0.75) rectangle (0.75,0.75);
			\draw [densely dashed,fill=white,rotate=45,xshift=2.8cm] (-0.75,-0.75) rectangle (0.75,0.75);
			\fill (3.9,0) circle[radius=3pt];
			\draw [densely dashed,-latex] (0,0) -- (3.9,0);
			\draw [densely dashed,-latex] (-3.9,0) -- (-1.5,0);
			\fill [rotate=45] (2.8,0) circle[radius=3pt];
			\draw [rotate=45,densely dashed,-latex] (0,0) -- (2.8,0);
			\fill (-3.9,0) circle[radius=3pt];
			\fill (-1.5,0) circle[radius=3pt];
			\draw (-1.5,0) node[anchor=north] {$s_*$};
			\draw [<->] (0,-6) -- (0,6);
			\fill (0,0) circle[radius=3pt];
			\draw (0,0) node[anchor =east] {$\scriptstyle 0$};
	    \end{tikzpicture}
	    \caption{Example~\ref{ex1}}
	    \label{fig:ex1}
	\end{figure}
	
	\noindent
	Along the $x$-axis, the translations of $\mathscr{O}$ are $\mathscr{O}_s = (s,0)+\mathscr{O} \subset \Omega $ for $|s|<R-\ell $. For $|s|<R-\ell $, let $\lambda _1(s)=\lambda _1(\Omega \setminus \mathscr{O}_s)$, the first Dirichlet eigenvalue of the $p$-Laplacian, for $\frac{2d+2}{d+2}<p<\infty $. Let
	\begin{align*}
	    s_*=\sup \,\Bigl\{s\in (-R+\ell , 0): \sigma _{H_s}(\{(x,y)\in \Omega : x<s\})\subset \Omega \Bigr\}.
	\end{align*}
	Now, applying Theorem~\ref{thm-mono}
	\begin{enumerate}[label=(\roman*)]
	    \item with $h=(1,0)$, we get $\lambda _1(s)$ is strictly decreasing for $s\in [0,R-\ell )$,
	    \item with $h=(-1,0)$, we get $\lambda _1(s)$ is strictly increasing for $s\in (-R+\ell , s_*)$.
	\end{enumerate}
	
	\noindent
	We get similar monotonicity results for the translations of $\mathscr{O}$ along the straight line $y=x$.
	\end{example}
	\begin{remark}\label{ex2}
	In Theorem~\ref{thm-circ}, we can consider the obstacle $\mathscr{O}$ of the form $\mathscr{O}=\bigcup \limits_{j=1}^{k} \overline{B}_{\rho _j} (z_j) \subset \Omega $, a finite union of closed balls, such that the centers $z_j$'s lie on the ray $a+\mathbb{R}^+ \eta $. Now, the rotations of the obstacle $\mathscr{O}$ about the point $a$ are given by 
	\begin{align*}
		\mathscr{O}_s \mathrel{\mathop:}= \bigcup \limits_{j=1}^k \overline{B}_{\rho _j} (a+R_{s}(-a+z_j)) \mbox{ for } s \in [-1,1], \mbox{ and } \mathrm{C}_\mathscr{O} \mathrel{\mathop:}= \Big\{s\in [-1 ,1] : \mathscr{O}_s \subset \Omega \Big\}.
	\end{align*}
	In this case, also, we have the same conclusions as Theorem~\ref{thm-circ}. 
	\end{remark}
	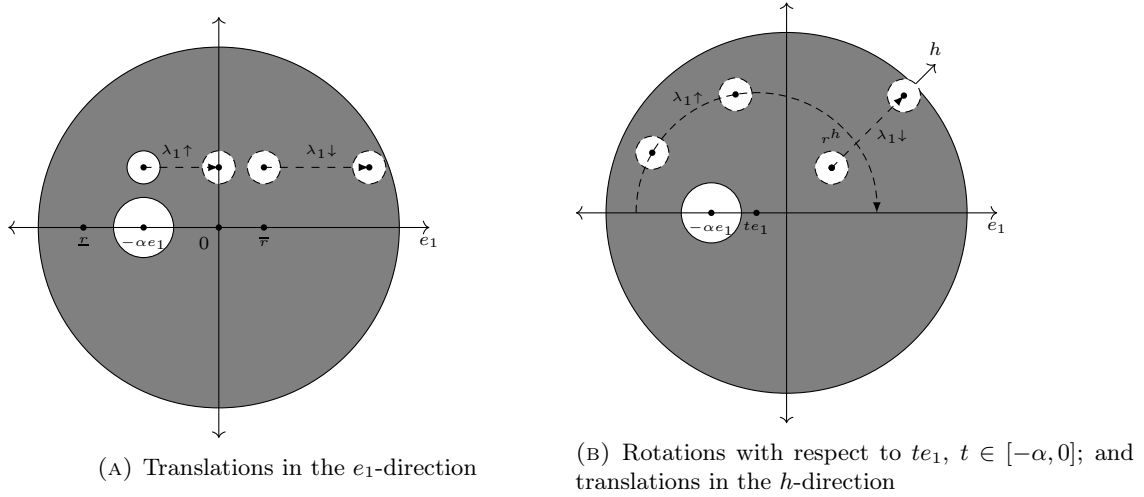
\begin{figure}[tbh!]
	    \centering
	    \begin{subfigure}{0.45\textwidth}
	    \begin{tikzpicture}[scale=0.4]
		\draw [fill=gray] (0,0) circle (6cm);
		\draw [fill=white] (-2.5,0) circle (1cm);
		\draw [fill=white] (-2.5,2) circle (0.55cm);
		\fill (-2.5,2) circle[radius=3pt];
		\draw [densely dashed,fill=white] (0,2) circle (0.55cm);
		\fill (0,2) circle[radius=3pt];
		\draw [densely dashed,fill=white] (1.5,2) circle (0.55cm);
		\fill (1.5,2) circle[radius=3pt];
		\draw [dashed,-latex] (-2.5,2) -- (0,2);
		\draw (-1.4,2) node[anchor=south] {$\scriptscriptstyle \lambda _1 \uparrow $};
		\draw [densely dashed,fill=white] (5,2) circle (0.55cm);
		\fill (5,2) circle[radius=3pt];
		\draw [dashed,-latex] (1.5,2) -- (5,2);
		\draw (3.4,2) node[anchor =south] {$\scriptscriptstyle \lambda _1 \downarrow $};
		\fill (-2.5,0) circle[radius=3pt];
		\draw (-2.5,0) node[anchor =north] {$\scriptscriptstyle -\alpha e_1$};
		\draw [<->] (0,-7) -- (0,7);
		\fill (0,0) circle[radius=3pt];
		\draw (0,0) node[anchor=north east] {$\scriptstyle 0$};
		\fill (1.5,0) circle[radius=3pt];
		\draw (1.5,0) node[anchor =north] {$\scriptscriptstyle \overline{r} $};
		\fill (-4.5,0) circle[radius=3pt];
		\draw (-4.5,0) node[anchor =north] {$\scriptscriptstyle \underline{r}$};
		\draw [<->] (-7,0) -- (7,0);
		\draw (7,0) node[anchor =north] {$\scriptstyle e_1$};
		\end{tikzpicture}
	    \caption{Translations in the $e_1$-direction}
	    \label{fig:subim1}
	    \end{subfigure}
	    \begin{subfigure}{0.45\textwidth}
	    \begin{tikzpicture}[scale=0.4]
		\draw [rotate=45,<->] (0,0) -- (7,0);
		\draw (45:7) node[anchor =south] {$\scriptstyle h$};
		\draw [fill=gray] (0,0) circle (6cm);
		\draw [fill=white] (-2.5,0) circle (1cm);
		\draw [fill=white,xshift=-1cm,densely dashed] (150:4) circle (0.55cm);
		\fill [xshift=-1cm] (150:4) circle[radius=3pt];
		\draw [xshift=-1cm,densely dashed, fill=white] (100:4) circle (0.55cm);
		\fill [xshift=-1cm] (100:4) circle[radius=3pt];
		\draw [-latex, densely dashed,xshift=-1cm] (180:4) arc (180:0:4);
		\draw [densely dashed,fill=white] (1.5,1.5) circle (0.55cm);
		\fill (1.5,1.5) circle[radius=3pt];
		\draw [xshift=-1cm] (125:4) node[anchor=south] {$\scriptscriptstyle \lambda _1 \uparrow $};
		\draw [densely dashed,fill=white] (3.9,3.9) circle (0.55cm);
		\fill (3.9,3.9) circle[radius=3pt];
		\draw [dashed,-latex] (1.5,1.5) -- (3.9,3.9);
		\draw (3.5,2) node[anchor =south] {$\scriptscriptstyle \lambda _1 \downarrow $};
		\fill (-2.5,0) circle[radius=3pt];
		\draw (-2.5,0) node[anchor =north] {$\scriptscriptstyle -\alpha e_1$};
		\draw [<->] (0,-7) -- (0,7);
		\fill (-1,0) circle[radius=3pt];
		\draw (-1,0) node[anchor =north] {$\scriptscriptstyle t e_1$};
		\fill (1.5,1.5) circle[radius=3pt];
		\draw (1.5,2) node[anchor =south] {$\scriptscriptstyle r^h$};
		\draw [<->] (-7,0) -- (7,0);
		\draw (7,0) node[anchor =north] {$\scriptstyle e_1$};
	\end{tikzpicture}
	    \caption{Rotations with respect to $t e_1$, $t\in [-\alpha ,0]$; and translations in the $h$-direction}
	    \label{fig:subim2}
	    \end{subfigure}
	    \caption{An eccentric annular domain with a hole.}
	    \label{fig-ecc}
	\end{figure}
	\noindent
	\subsection{An eccentric annular domain with a spherical hole}\label{ecc_ann}
	Let $\frac{2d+2}{d+2}<p<\infty $. For given $0<r<R$ and $0\leq \alpha <R-r,$	we consider an eccentric annular domain $\Omega = B_R(0)\setminus \overline{B_r}(-\alpha e_1)\subset \mathbb{R}^d$ (see Figure~\ref{fig-ecc}). Let $\rho >0$ be such that $\overline{B}_{\rho }(y)\subset \Omega$ for some $y\in \mathbb{R}^d .$ Now define
	\begin{align*}
	    \Omega _\rho :=\Bigl\{y\in \Omega : \overline{B}_{\rho }(y)\subset \Omega \Bigr\}, \mbox{ and } 
	    \lambda _1(y):=\lambda _1\left(\Omega \setminus \overline{B}_{\rho }(y)\right) \mbox{ for } y\in \Omega _\rho .
	\end{align*}
	We want to study the behaviour of $\lambda _1 (\cdot )$ on $\Omega _\rho $. It is easy to observe that
	\begin{enumerate}[label=(\alph*)]
	    \setcounter{enumi}{0}
	    \item $\overline{B}_{\rho }(y)$ is Steiner symmetric with respect to any affine-hyperplane through $x$; 
	    \item $\overline{B}_\rho (y)$ is foliated Schwarz symmetric with respect to $a+\mathbb{R}^+ (y-a)$ for any $a\in \mathbb{R}^d$;
	    \item $\Omega $ is foliated Schwarz symmetric with respect to $t e_1+\mathbb{R}^+ e_1$ for $t\in [-\alpha ,0]$;
	    \item the sets  $\{x\in \Omega : x_1<\underline{r} \}\bigcup \sigma _{\scriptscriptstyle H_{\underline{r}}}(\{x\in \Omega : x_1<\underline{r} \})$ and  $\{x\in \Omega : x_1>\overline{r}\}\bigcup \sigma _{\scriptscriptstyle H_{\overline{r}}}(\{x\in \Omega : x_1>\overline{r}\})$ are convex in the $e_1$-direction, where $\overline{r}=\frac{R+r-\alpha }{2}$, $\underline{r}=-\frac{R+r+\alpha }{2}$.
	\end{enumerate}
	For $y\in \mathbb{R}^d$, we write $y=(s,z)\in \mathbb{R}\times   \mathbb{R}^{d-1}.$ Now for a given $ z\in \mathbb{R}^{d-1} , \beta >0,$ we consider the sets
	\begin{align*}
	    &\mathrm{L}_{z}:=\Bigl\{s\in (-R,R): (s,z)\in \Omega _\rho \Bigr\};\\
        &\mathrm{S}_{\beta }(t e_1):=\Omega _\rho \cap \partial B_\beta (t e_1),\quad t\in [-\alpha ,0].
    \end{align*}
    \begin{remark} \label{remark1}
    Let $(s,z_1) \in \Omega_\rho $. Then we have the following:
	\begin{enumerate}[label=(\roman*)]
	   	     \item Using  the axial symmetry of $\Omega$, we obtain $\lambda _1(s,z)=\lambda _1(s,z_1),$ for $z\in \mathbb{R}^{d-1}$ such that $(s,z)\in \Omega_\rho$ and $|z|=|z_1|.$
	        \item From (a), (d) and  Theorem~\ref{thm-mono} with $h=-e_1$ (and $h=e_1$) , we get
	        \begin{align*}
	            \lambda _1(\cdot ,z_1) \mbox{ is strictly increasing on } \mathrm{L}_{z_1} \bigcap \left.\left(-\sqrt{(R-\rho )^2-|z_1| ^2}, \underline{r}\right.\right],
	        \end{align*}
	        and 
	        \begin{align*}
	            \lambda _1(\cdot , z_1) \mbox{ is strictly decreasing on } \mathrm{L}_{z_1}\bigcap \left[\left.\overline{r}, \sqrt{(R-\rho )^2-|z_1| ^2} \right)\right. .
	        \end{align*}
	        \item If $(s_1,z_1)$ and $(s_2,z_2)\in \Omega_\rho$  such that $s_1< s_2$ and 
	        $(s_1-t)^2+|z_1|^2=(s_2-t)^2+|z_2|^2$ for some $t\in [-\alpha,0]$, then by Theorem~\ref{thm-circ}, we get
	        \begin{align*}
	          \lambda_1(s_1,z_1)<\lambda_1(s_2,z_2).
	        \end{align*}
	        In particular,  for $s_1,s_2\in [-\alpha , 0]$ with $s_1<s_2$, by taking $t=\frac{s_1+s_2}{2}$ we obtain
	        \begin{align*}
	            \lambda _1(\cdot ,z) \mbox{ is strictly increasing on } \mathrm{L}_{z }\cap [-\alpha , 0].
	        \end{align*}
	\end{enumerate}
	\end{remark}
    \begin{remark}
	   In general, for any $h=(h_1,h')\in \mathbb{S}^{d-1},$ let 
	   \begin{align*}
	       r^h:=\frac{\sqrt{R^2-\alpha ^2 |h'|^2}+r-\alpha h_1}{2}.
	   \end{align*}
	   Then, the set $\{x\in \Omega : x\cdot h>r^h\}\bigcup \sigma _{\scriptscriptstyle H_{r_h}}(\{x\in \Omega : x\cdot h>r^h\})$ is convex in the $h$-direction. For  $y\in \Omega _\rho $, define $\mathrm{L}_{y}=\Bigl\{s\in [0,R) : y+s h \in \Omega _\rho \Bigr\}$. Now, from (a) and Theorem~\ref{thm-mono} we get
	   \begin{align*}
	       \lambda _1(y+s h) \mbox{ is strictly decreasing for } s\in \mathrm{L}_y \cap [r^h, R).
	   \end{align*}
	   \end{remark}
	   \begin{remark}
	   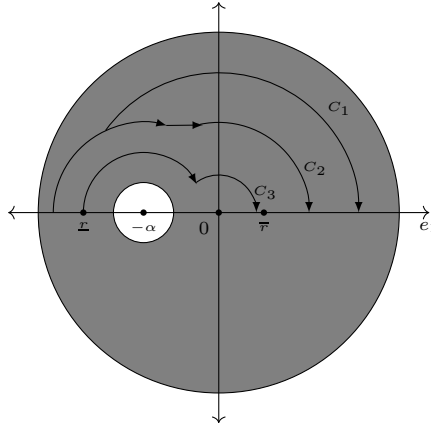
\begin{figure}[hbt!]
	    \centering 
	    \begin{tikzpicture}[scale=0.4]
		\draw [fill=gray] (0,0) circle (6cm);
		\draw [fill=white] (-2.5,0) circle (1cm);
		\draw [xshift=-2.5cm,-latex] (180:2) arc (180:28:2);
		\draw [-latex] (127:1.25) arc (127:0:1.25);
		\draw (3.2,2) node[anchor =north] {$\scriptscriptstyle C_2$};
		\draw [xshift=-2.5cm,-latex] (180:3) arc (180:75:3);
		\draw [-latex, xshift=-2.52cm] (75:3)--(2,2.912);
		\draw [-latex] (102:3) arc (102:0:3);
		\draw [-latex](144:4.65) arc (144:0:4.65);
		\draw (4,4) node[anchor =north] {$\scriptscriptstyle C_1$};
		\draw (1.55,1.2) node[anchor =north] {$\scriptscriptstyle C_3$};
		\fill (-2.5,0) circle[radius=3pt];
		\draw (-2.5,-0) node[anchor =north] {$\scriptscriptstyle -\alpha $};
		\draw [<->] (0,-7) -- (0,7);
		\fill (0,0) circle[radius=3pt];
		\draw (0,0) node[anchor=north east] {$\scriptstyle 0$};
		\fill (1.5,0) circle[radius=3pt];
		\draw (1.5,0) node[anchor =north] {$\scriptscriptstyle \overline{r} $};
		\fill (-4.5,0) circle[radius=3pt];
		\draw (-4.5,0) node[anchor =north] {$\scriptscriptstyle \underline{r}$};
		\draw [<->] (-7,0) -- (7,0);
		\draw (7,0) node[anchor =north] {$\scriptstyle e_1$};
	\end{tikzpicture}
	\caption{Monotonicity along certain paths in $\Omega $.}
	\label{fig-ecc1}
	\end{figure}
	   Let $C:(-R,R)\longrightarrow \Omega $ be a continuous path in $\Omega $ such that (see Figure~\ref{fig-ecc1}):
	   \begin{itemize}
	       \item on $(-R, -\alpha )$, $C$ is a circular arc centered at $t_1 e_1$ with $t_1\in [-\alpha ,0]$;
	       \item on $[-\alpha ,0]$, $C$ is either a circular arc centered at $t_2 e_1$ with $t_2\in [-\alpha ,0]$ or a line segment parallel to the $e_1$-axis;
	       \item on $(0, R)$, $C$ is a circular arc centered at $t_3 e_1$ with $t_3\in [-\alpha ,0]$.
	   \end{itemize}
	   Now, from Remark \ref{remark1}, we see that $\lambda _1(\cdot )$ is strictly increasing along the path $C$; i.e., for any $(s_1,z_1), (s_2,z_2)\in C$ with $s_1<s_2$ we have $\lambda _1(s_1,z_1) < \lambda _1(s_2, z_2).$
	\end{remark}
    \begin{remark}[\textbf{Optimal placement of the obstacle}]
    For $y\in \Omega _\rho $, we have $y, |y|e_1 \in \mathrm{S}_{|y|}(0)$. If $y_1 < |y|$,  then  using  Theorem~\ref{thm-circ} we obtain
    $\lambda _1(y)<\lambda (|y|e_1)$. Thus
    \begin{align*}
        \sup \left\{\lambda _1(y): y\in \Omega _\rho \right\}= \sup \left\{\lambda _1(s,0): s\in \mathrm{L}_{0}\cap [0,\overline{r}] \right\}.
    \end{align*}
    If $0<\rho <\alpha -r$, then $0, \overline{r} \in \mathrm{L}_0$, and hence by (iii) of Remark \ref{remark1} we get
    \begin{align*}
        \sup \left\{\lambda _1(y): y\in \Omega _\rho \right\} &=\max \left\{\lambda _1(s,0): s\in [0,\overline{r}] \right\}.
    \end{align*}
    On the other hand, if  $\alpha \leq r$ or $\rho >\alpha -r$, then $0\notin \mathrm{L}_0$. Thus, the above arguments fail to conclude that the supremum is attained in $\Omega_\rho.$ However, from a \emph{Mathematica~12} plot of $\lambda _1(\cdot )$ on $[0,R)\cap L_0$ (for the various values of $\alpha, r$, and $\rho$), we observed that the maximum is attained at a unique point in $(0,\overline{r})\cap L_0$. Giving an analytic explanation of this behaviour of $\lambda_1 (\cdot,0)$ in  $[0,\overline{r}]\cap \mathrm{L}_{0}$ seems to be  an interesting  problem to explore. 
    \end{remark}

    \subsection{\bf The symmetries of the first eigenfunctions:}
    In this subsection, we take $p\in(1,\infty)$ and $\Omega _{\rm out}\setminus \overline{\Omega _{\rm in}} \subset \mathbb{R}^d$ is a domain as given in~\ref{hypothesis}. We establish that the  first eigenfunctions of \eqref{eigen1} and \eqref{eigen2-1} inherit some of the symmetries of the underlying domains. 
    
    \begin{remark}\label{rmk-1.5}
        Let $H\in \mathscr{H}_{\rm ad}$ be such that $P_H(\Omega _{\rm out}\setminus \overline{\Omega _{\rm in}})=\Omega _{\rm out}\setminus \overline{\Omega _{\rm in}}$. 
	    \begin{enumerate}[label=(\roman*)]
	        \item Let $u$ be an eigenfunction corresponding to the first eigenvalue $\nu _1(\Omega _{\rm out}\setminus \overline{\Omega _{\rm in}})$ of~\eqref{eigen1}. Assume that  $\sigma _H(\Omega _{\rm in})=\Omega _{\rm in}$. If $u$ is positive, then from Proposition~\ref{propo:pol_Lip1} and \eqref{eqn:pol_norm}, we see that $P_H(u)$ is also an eigenfunction corresponding to $\nu _1(\Omega _{\rm out}\setminus \overline{\Omega _{\rm in}})$. Since the norms of $u$ and $P_H(u)$ are same, by the simplicity of $\nu _1$, we get $P_H(u)=u.$  If $u$ is negative then, we get $P^H(u)=-P_H(-u)=u$.
	        \item Similarly, if $\sigma _H(\Omega _{\rm out})=\Omega _{\rm out}$ and $v$ is a positive eigenfunction corresponding to the first eigenvalue $\tau _1(\Omega _{\rm out}\setminus \overline{\Omega _{\rm in}})$ of~\eqref{eigen2}, then $P_H(v)=v$  and $P^H(-v)=-v$.
	    \end{enumerate}
	\end{remark}
	\begin{definition}
  Let $\Omega$ be foliated  Schwarz symmetric with respect to the ray $a+\mathbb{R}^+\eta $.	Then a function $u:\Omega\longrightarrow \mathbb{R}$ is said to  be foliated Schwarz symmetric with respect to the same ray, if $P_H(u)=u$ for every $H\in \mathscr{H}_{a,\eta }$ (see~\cite[Lemma~6.3]{BrockSolynin2000} and~\cite[Section~3]{VanJeanWillem08}).
	\end{definition}
    \begin{remark}
        Let $\Omega _{\rm out}\setminus \overline{\Omega _{\rm in}}$ be foliated Schwarz symmetric with respect to the ray $a+\mathbb{R}^+\eta $ for some $a\in \mathbb{R}^d$ and $\eta \in \mathbb{S}^{d-1}$. Then, we have $P_H(\Omega _{\rm out}\setminus \overline{\Omega _{\rm in}})=\Omega _{\rm out}\setminus \overline{\Omega _{\rm in}}$ for every $H\in \mathscr{H}_{a,\eta }$. 
        \begin{enumerate}[label=(\roman*)]
            \item  Assume that $\Omega _{\rm in}=B_r(a)$ for some $r\geq 0$. Then $\sigma _H(\Omega _{\rm in})=\Omega _{\rm in}$ for every $H\in \mathcal{H}_{a,\eta }$. Hence, for any positive eigenfunction $u$ corresponding to $\nu _1(\Omega _{\rm out}\setminus \overline{\Omega _{\rm in}})$, from Remark~\ref{rmk-1.5}, we obtain $P_H(u)=u$ for every $H\in \mathcal{H}_{a,\eta }.$ Thus, $u$ is foliated Schwarz symmetric with respect to the ray $a+\mathbb{R}^+ \eta .$
            \medskip
            \item Similarly, Assume that $\Omega _{\rm out}=B_R(a)$ for some $R>0$. Then, any positive eigenfunction $u$ corresponding to the first eigenvalue $\tau_1(\Omega _{\rm out}\setminus \overline{\Omega _{\rm in}})$ of~\eqref{eigen2} is foliated Schwarz symmetric with respect to  $a+\mathbb{R}^+\eta $.  
        \end{enumerate}
    \end{remark}
    \begin{remark}
        Let $\Omega = B_R(0)\setminus \overline{B}_r(-\alpha e_1)\subset \mathbb{R}^d$ be the eccentric annular domain as given in Subsection \ref{ecc_ann}. Then $\Omega$ is foliated Schwarz symmetric with respect to $t e_1+\mathbb{R}^+ e_1$ for $t\in [-\alpha ,0]$. 
        \begin{enumerate}[label=(\roman*)]
            \item If, we take $\Omega _{\rm out}=\Omega $ and $\Omega _{\rm in}=\emptyset $, then any positive eigenfunction corresponding to the first Dirichlet eigenvalue $\lambda _1(\Omega )$ is foliated Schwarz symmetric with respect to $t e_1+\mathbb{R}^+ e_1$ for every $t\in [-\alpha ,0]$.
            \item If, we take  $\Omega _{\rm out}=B_R(0)$ and $\Omega _{\rm in}=B_r(-\alpha e_1)$, then 
            \begin{itemize}
                \item any positive eigenfunction corresponding to the first eigenvalue $\nu _1(\Omega )$ of~\eqref{eigen1} is foliated Schwarz symmetric with respect to the ray $-\alpha e_1 +\mathbb{R}^+ e_1$;
                \item any positive eigenfunction corresponding to the first eigenvalue $\tau _1(\Omega )$ of~\eqref{eigen2} is foliated Schwarz symmetric with respect to the ray $\mathbb{R}^+ e_1$.
            \end{itemize}
        \end{enumerate}
    \end{remark}
   
    \bibliography{refs.bib}
    \bibliographystyle{abbrvnat}
\end{document}